\documentclass[11pt]{article}
\usepackage{mathrsfs}
\usepackage{amsmath}
\usepackage{amsthm}
\usepackage{amssymb}
\usepackage{amsfonts}
\usepackage{color,xcolor}
\usepackage{graphicx}
\usepackage{geometry}
\usepackage{manfnt}
\usepackage[pagewise]{lineno}
\usepackage[colorlinks,citecolor=blue,urlcolor=blue]{hyperref}
\newtheorem{theorem}{Theorem}[section]

\newtheorem{lemma}{Lemma}[section]

\newtheorem{remark}{Remark}[section]

\def\geq{\geqslant}\def\leq{\leqslant}

\begin{document}
\title{ \bf SDEs with subcritical Lebesgue--H\"{o}lder drift and driven by $\alpha$-stable processes}

\author{Rongrong Tian$^{a,b}$ and Jinlong Wei$^c$\\ {\small \it $^a$School of Mathematics and Statistics, Wuhan University of Technology}  \\ {\small \it  Wuhan, 430070, P.R.China}
\\ {\small \it $^b$Institut de Math\'{e}matiques de Bourgogne, Universit\'{e} Bourgogne Europe}
  \\ {\small \it  Dijon, 21000, France}  \\ {\small \tt  tianrr2018@whut.edu.cn} \\ {\small \it $^c$School of Statistics and Mathematics, Zhongnan University of} \\  {\small \it Economics and Law, Wuhan, 430073, P.R.China}  \\ {\small \tt   weijinlong.hust@gmail.com}}

\date{}
\maketitle
\noindent{\hrulefill}
\vskip1mm\noindent
{\bf Abstract} We obtain the unique weak and strong solvability for time inhomogeneous stochastic differential equations  with the  drift in subcritical Lebesgue--H\"{o}lder spaces $L^p([0,T];{\mathcal C}_b^{\beta}({\mathbb R}^d;{\mathbb R}^d))$ and driven by $\alpha$-stable processes for $\alpha\in (0,2)$. The weak well-posedness is derived for $\beta\in (0,1)$, $\alpha+\beta>1$ and $p>\alpha/(\alpha+\beta-1)$ through Prohorov's theorem, Skorohod's representation and the regularity estimates of solutions for a class of fractional Kolmogorov equations. The pathwise uniqueness and Davie's type uniqueness are proved for $\beta>1-
\alpha/2$ by using It\^{o}--Tanaka's trick. Moreover, we give a counterexample to the pathwise uniqueness for the supercritical Lebesgue--H\"{o}lder drift with $\alpha\in (0,2)$ to explain that the present result is sharp.


 \vskip0mm\noindent{\hrulefill}

\section{Introduction}\label{sec1}\setcounter{equation}{0}
Consider the following stochastic differential equation (SDE for short)
\begin{equation}\label{1.1}
X_t=x+\int_0^tb(s,X_s)ds+L_t, \quad t\in [0, T], \quad x\in {\mathbb R}^d,
\end{equation}
where $T>0$ is a fixed final time horizon and the dimension $d\geq 1$. The drift coefficient $b:[0,T]\times {\mathbb R}^d
\rightarrow {\mathbb R}^d$ is Borel measurable, and $\{L_t\}_{t\in [0,T]}$ is a $d$-dimensional L\'{e}vy process on a given filtered probability space $(\Omega,{\mathcal F},\{{\mathcal F}_t\}_{t\in[0,T]},{\mathbb P})$, which satisfies the usual hypotheses of completeness and right continuity.

When $\{L_t\}_{t\in[0,T]}$ is a $d$-dimensional standard Brownian motion, the unique strong solvability of \eqref{1.1} with bounded drift was first established by Zvonkin \cite{Zvo} in the one-dimensional case ($d=1$), and later extended by Veretennikov \cite{Ver} to general dimensions $d>1$. The Sobolev differentiability of solutions in the spatial variables was obtained by Mohammed, Nilssen and Proske \cite{MNP}. More recently, Wei, Lv and Wang \cite{WLW} proved that the unique strong solution generates a stochastic flow of quasi-diffeomorphisms when the drift $b$ is Dini continuous in space. For H\"{o}lder continuous drifts, Flandoli, Gubinelli and Priola \cite{FGP1} established the existence of stochastic flows of diffeomorphisms. Their result was subsequently generalized in two directions: by Tian, Ding and Wei \cite{TDW}, and Wei, Duan, Gao and Lv \cite{WDGL} to the case of time-integrable drifts with exponent $q\geq 2$, and by Galeati and Gerencs\'{e}r \cite{GG}, as well as Wei, Hu and Yuan \cite{WHY1}, to the case $q\leq 2$, using different approaches.

If the drift is unbounded but exhibits spatial growth, strong well-posedness was established by Fang, Imkeller and Zhang \cite{FIZ}, Fang and Zhang \cite{FZ}, and also by Flandoli, Gubinelli and Priola \cite{FGP2}. In the case where $b$ belongs to the Krylov--R\"{o}ckner class, strong well-posedness for (\ref{1.1}) was first obtained by Krylov and R\"{o}ckner \cite{KR}. This was later extended to the local Krylov--R\"{o}ckner class by Xia, Xie, Zhang and Zhao \cite{XX}, and further to the critical case by Beck, Flandoli, Gubinelli and Maurelli \cite{BFGM}, Fang and Tian \cite{FT}, Kinzebulatov and Madou \cite{KM1,KM2}, Kinzebulatov, Semenov and Song \cite{KSS}, Krylov \cite{Kry21-1,Kry21-2,Kry21-3,Kry23-1,Kry23-2,Kry25}, Nam \cite{Nam}, R\"{o}ckner and Zhao \cite{RZ}, and Wei, Lv and Wu \cite{WLWu}. For more details on the case of nonconstant diffusion coefficients, we refer to Zhang and Yuan \cite{ZY}, and Zhang \cite{Zha05,Zha11}.

When $\{L_t\}_{t\in [0,T]}$ is non-Gaussian, the drift is time-independent and $d=1$, Tanaka, Tsuchiya and Watanabe \cite{TTW} first proved strong existence and uniqueness for the symmetric $\alpha$-stable process with $\alpha\in (0,2)$ and a H\"{o}lder continuous drift $b$ whose H\"{o}lder index exceeds $1-\alpha$. Since then, their result has been extended in several directions. For $d=1$, Mytnik and Weinberger \cite{MW} established strong well-posedness for drifts in the Kato class, while Kurenok \cite{Kurenok} obtained weak existence for time-dependent bounded drifts. In the case of nonconstant diffusion coefficients with time-independent $b$, weak existence and/or pathwise uniqueness have also been proved in \cite{Fournier,HW,LZ,XZZ}.

More recently, many researchers have studied the case $d>1$. For time-dependent $b$, Priola  \cite{Priola18} proved unique strong solvability of \eqref{1.1} when the drift is bounded (in time and space) and H\"{o}lder continuous (in space), and the driving process is of $\alpha$-stable type with $\alpha\in (0,2)$. Similar results for time-independent drift were established in \cite{Priola12,Priola15}. Later, Chen, Zhang and Zhao \cite{CZZ} extended these results to nonconstant diffusion coefficients and $\alpha$-stable-like processes. We also refer to \cite{HP} for the Malliavin differentiability of strong solutions, \cite{CSZ} for the existence of stochastic flows of diffeomorphisms, \cite{HW} for results with Besov--H\"{o}lder drifts.

For time-independent drifts belonging to certain Kato classes, Chen, Kim and Song \cite{CKS}, and Kim and Song \cite{KS} established weak well-posedness, while Bogachev and Pilipenko \cite{BP} proved  strong existence.  More recent generalizations include Athreya, Butkovsky and Mytnik \cite{ABM} (for distributional drifts); De Raynal, Menozzi and Priola \cite{DMP1} (for bounded and continuous drifts); and De Raynal and Menozzi \cite{DM}, Kremp and Perkowski \cite{KP}, and Song and Xie \cite{SX} (for Besov drifts). For further developments in this direction, see also \cite{Kulik,TWD,Zhang13}.

In the case where the drift is integrable only in time and $\{L_t\}_{t\in [0,T]}$ is a symmetric rotationally invariant $\alpha$-stable process, Tian and Wei \cite{TW} proved strong existence and pathwise uniqueness for $b\in L^p([0,T];{\mathcal C}^\beta_b({\mathbb R}^d;\mathbb{R}^d))$ with $\alpha\in [1,2)$ and $p>\max\{\alpha/(\alpha-1),2\alpha/(\alpha+2\beta-2)\}$. Butkovsky and Gallay \cite{BG} further obtained weak existence for $b\in L^p([0,T];L^q({\mathbb R}^d;{\mathbb R}^d))$ with $\alpha\in (1,2)$ and $(\alpha-1)/p+d/q<\alpha-1$. However, only limited progress has been made on the strong well-posedness of \eqref{1.1} for $\alpha\in(0,1)$ with time-integrable drifts.

In this paper, we first prove weak existence of \eqref{1.1} for $\alpha\in(0,2)$ and time-integrable drifts by combining Prohorov's theorem and Skorohod's representation. We then apply either It\^{o}--Tanaka's trick or Zvonkin's transformation to establish uniqueness in law, pathwise uniqueness, and Davie's type uniqueness.

\subsection{Setup and notations}
We provide a list of the main notations and conventions adopted throughout the paper.

$\bullet$ ${\mathbb N}$ denote the set of positive integers and ${\mathbb N}_0=\mathbb{N}\cup \{0\}$. $\mathbb{R}_+$ denotes the set of all positive real numbers. a.s. and a.e. are the abbreviations of almost surely and almost everywhere, respectively. For every $R>0$, $B_R:=\{x\in{\mathbb R}^d:|x|\leq R\}$.

$\bullet$  $\nabla$ denotes the gradient of a function with respect to the spatial variables.  For a given ${\mathbb R}^{n\times m}$ matrix-valued function $\Xi$ with $n,m\in {\mathbb N}$, $\|\Xi\|$ denotes its Hilbert--Schmidt norm.

$\bullet$ $\mathcal{C}_{bu}(\mathbb{R}^d)$ is the space consisting of all bounded uniformly continuous functions in $\mathbb{R}^d$. For $h\in \mathcal{C}_{bu}(\mathbb{R}^d)$, we set
\[
\|h\|_{{\mathcal C}_{bu}({\mathbb R}^d)}=\sup_{x\in {\mathbb R}^d}|h(x)|=:\|h\|_0.
\]

$\bullet$ For $\beta\in (0,1)$, the H\"{o}lder space ${\mathcal C}^\beta_b({\mathbb R}^d)$ consists of all bounded $\beta$-H\"{o}lder continuous functions. For $h\in{\mathcal C}^\beta_b({\mathbb R}^d)$, its norm is given by
\[
\|h\|_{{\mathcal C}^\beta_b({\mathbb R}^d)}=\sup_{x\in {\mathbb R}^d}|h(x)|+\sup_{x,y\in{\mathbb R}^d, x\neq y}\frac{|h(x)-h(y)|}{|x-y|^\beta} =:\|h\|_0+[h]_{\beta}=:\|h\|_\beta.
\]
By  \cite[Proposition 7, p.142]{Stein}, $h\in{\mathcal C}_b^\beta({\mathbb R}^d)$ if and only if $h\in{\mathcal C}_b({\mathbb R}^d)$ and there exists a positive constant $A$ such that
\[
\|\partial_\xi P_\xi h\|_0=\sup_{x\in{\mathbb R}^d}|\partial_\xi P_\xi h(x)|\leq A\xi^{-1+\beta},  \quad \forall \ \xi\in {\mathbb R}_+,
\]
where
\begin{equation}\label{1.2}
P_\xi h(x)=\frac{\Gamma(\frac{d+1}{2})}{\pi^{\frac{d+1}{2}}}\int_{{\mathbb R}^d}\frac{\xi h(x-z)}{(\xi^2+|z|^2)^{\frac{d+1}{2}}}dz, \quad \forall \ \xi\in {\mathbb R}_+,
\end{equation}
and $\Gamma$ is the Gamma function. Additionally, the norms $\|h\|_0+\sup_{\xi>0}[\xi^{1-\beta}\|\partial_\xi P_\xi h\|_0]$
and $\|h\|_\beta$ are equivalent to each other. Furthermore, if $\nabla^j h$ ($j$-th order gradient with respect to the spatial variables) is bounded and continuous for $j=0,1,2,\ldots,k\in {\mathbb N}$, and $[\nabla^k h]_\beta$ is finite, then we say $h\in {\mathcal C}_b^{k+\beta}({\mathbb R}^d)$.  For $h\in {\mathcal C}_b^{k+\beta}({\mathbb R}^d)$, its norm is defined by
\begin{equation*}
 \begin{split} \|h\|_{{\mathcal C}_b^{k+\beta}({\mathbb R}^d)}=&\;\sum_{j=0}^k\sup_{x\in{\mathbb R}^d}|\nabla^jh(x)|+\sup_{x,y\in{\mathbb R}^d, x\neq y}\frac{|\nabla^kh(x)-\nabla^kh(y)|}{|x-y|^{\beta}}
\nonumber\\ =&\;\sum_{j=0}^k\|\nabla^jh\|_0+[\nabla^kh]_\beta=:\|h\|_{k+\beta},
\end{split}
\end{equation*}
which is equivalent to
\[
\sum_{j=0}^k\|\nabla^jh\|_0+\sup_{\xi>0}\Big[\xi^{1-\beta}\|\partial_\xi P_\xi \nabla^kh\|_0\Big]=\sum_{j=0}^k\|\nabla^jh\|_0+\sup_{(\xi,x)\in {\mathbb R}_+\times {\mathbb R}^d }|\xi^{1-\beta}\partial_\xi P_\xi \nabla^kh(x)|.
\]

$\bullet$ For $p\in [1,\infty]$ and $k\in {\mathbb N}_0$, the Lebesgue--H\"{o}lder space $L^p([0,T];{\mathcal C}_b^{k+\beta}({\mathbb R}^d))$ denotes the space of $L^p([0,T])$-functions in time with values in ${\mathcal C}_b^{k+\beta}({\mathbb R}^d)$. For $g$ in this space, its norm is defined by
\begin{equation}\label{1.3}
 \begin{split}\|g\|_{L^p([0,T];{\mathcal C}_b^{k+\beta}({\mathbb R}^d))}=&\;\bigg[\sum_{j=0}^k\int_0^T \|\nabla^jg(t,\cdot)\|_0^pdt+\int_0^T[\nabla^k g(t,\cdot)]_{\beta}^pdt\bigg]^{\frac{1}{p}} \\ =:&\;\bigg[\sum_{j=0}^k\|\nabla^jg\|_{p,0}^p+[\nabla^kg]_{p,\beta}^p\bigg]^{\frac{1}{p}}
=:\|g\|_{p,k+\beta},
\end{split}
\end{equation}
which is equivalent to
\begin{equation}\label{1.4}
\begin{split} &\bigg[\sum_{j=0}^k\|\nabla^jg\|_{p,0}^p+\int_0^T\sup_{\xi>0}\|\xi^{1-\beta}\partial_\xi P_\xi \nabla^kg(t,\cdot)\|_0^p dt\bigg]^{\frac{1}{p}}\\ =&\;\bigg[\sum_{j=0}^k\|\nabla^jg\|_{p,0}^p+\int_0^T\sup_{(\xi,x)\in {\mathbb R}_+\times {\mathbb R}^d }|\xi^{1-\beta}\partial_\xi P_\xi \nabla^kg(t,x)|^p dt\bigg]^{\frac{1}{p}},
\end{split}
\end{equation}
where the integrals in \eqref{1.3} and \eqref{1.4} are interpreted as the essential supremum when $p=\infty$. The definition of the analogues spaces for $\mathbb{R}^d$ or $\mathbb{R}^{d\times d}$-valued functions is simply understood coordinate-wise. For simplicity of notation, if $f\in L^p([0,T];{\mathcal C}_b^{k+\beta}({\mathbb R}^d;{\mathbb R}^d))$, we also use $\|f\|_{L^p([0,T];{\mathcal C}_b^{k+\beta}({\mathbb R}^d))}$ or $\|f\|_{p,k+\beta}$ to denote its norm instead of the more precise but cumbersome $\|f\|_{L^p([0,T];{\mathcal C}_b^{k+\beta}({\mathbb R}^d;{\mathbb R}^d))}$; the same convention applies to other vector-valued or matrix-valued functions.

$\bullet$ For $\gamma\in (0,1)$, define the Bessel potential $\mathscr{I}_\gamma$ by $\mathscr{I}_\gamma=(I-\Delta)^{-\frac{\gamma}{2}}$. By \cite[Theorem 4, p.149]{Stein}, if $\beta+\gamma\notin \mathbb{N}$, then $\mathscr{I}_\gamma$ is an isomorphism from ${\mathcal C}_b^\beta({\mathbb R}^d)$ to ${\mathcal C}_b^{\beta+\gamma}({\mathbb R}^d)$; consequently, $\mathscr{I}_{-\gamma}$ is an isomorphism from ${\mathcal C}_b^{\beta+\gamma}({\mathbb R}^d)$ back to ${\mathcal C}_b^\beta({\mathbb R}^d)$. Similarly, for every $p\in [1,\infty]$ with $\beta+\gamma\notin \mathbb{N}$, $\mathscr{I}_\gamma$ maps $L^p([0,T];{\mathcal C}_b^\beta({\mathbb R}^d))$ isomorphically onto $L^p([0,T];{\mathcal C}_b^{\beta+\gamma}({\mathbb R}^d))$, and  $\mathscr{I}_{-\gamma}$ provides the inverse mapping.

$\bullet$ If $\{L_t\}_{t\in [0,T]}$ is a symmetric rotationally invariant $\alpha$-stable process on a filtered probability space $(\Omega,{\mathcal F},\{{\mathcal F}_t\}_{t\in[0,T]},{\mathbb P})$ with $\alpha \in (0,2)$, then the infinitesimal generator of its Markov semigroup is the fractional Laplacian $\Delta^{\frac{\alpha}{2}}=-(-\Delta)^{\frac{\alpha}{2}}$,  defined for sufficiently regular function $h$  by
\[
\Delta^{\frac{\alpha}{2}}h(x)=c_0\lim_{\varepsilon\downarrow
0}\int_{|y-x|>\varepsilon}\frac{h(y)-h(x)}{|x-y|^{d+\alpha}}dy
=c_0P.V.\int_{{\mathbb R}^d}\frac{h(y)-h(x)}{|x-y|^{d+\alpha}}dy,
\]
where $c_0=c_0(d,\alpha)$ is a constant depending only on $d$ and
$\alpha$.

$\bullet$ Let $K(t,x)$ be the heat kernel of the fractional Laplacian $\Delta^{\frac{\alpha}{2}}$. For all $m,n\in {\mathbb N}_0$, we have (see \cite[Lemma 2.2]{CZ})
\begin{equation}\label{1.5}
|\partial^m_t\nabla_x^nK(t,x)|\leq Ct(t^{\frac{1}{\alpha}}+|x|)^{-d-\alpha-n-\alpha m}, \quad t>0, \ x\in{\mathbb R}^d,
\end{equation}
where $C>0$ is a constant independent of $t$ and $x$.

\smallskip
$\bullet$ Let $N$ denote the Poisson random measure of the $\alpha$-stable process $L_t$. For $A\in \mathscr{B}(B_1\backslash\{0\})$,
\[
N((0,t],A)=\sum_{0<s\leq t}1_{A}(\Delta L_s)=\sharp\{0<s\leq t: \Delta L_s=L_s-L_{s-}\in A\},
\]
where $\mathscr{B}(B_1\backslash\{0\})$ denotes the Borel $\sigma$-algebra of $B_1\backslash\{0\}$.
The compensated Poisson random measure $\tilde{N}$ is defined by
\[
\tilde{N}((0,t],A)=N((0,t],A)-\nu(A)t, \ A\in \mathscr{B}(B_1\backslash\{0\}),
\]
where $\nu(dx)=c(d,\alpha)|x|^{-d-\alpha}dx$ is the L\'{e}vy measure. According to the L\'{e}vy--It\^{o} decomposition \cite[Theorem 2.4.16]{App},
\begin{equation}\label{1.6}
L_t=\int_0^t\int_{0<|z|< 1}z\tilde{N}(ds,dz)+\int_0^t\int_{|z|\geq 1}zN(ds,dz).
\end{equation}

$\bullet$ The letter $C$  denotes a positive constant, which may vary from line to line. For a parameter or a function $\zeta$, $C(\zeta)$ means the constant depends only on $\zeta$; when no confusion arises, it may be simply written as $C$. For two constants or functions $A_1$ and $A_2$, we use $A_1\wedge A_2:=\min\{A_1,A_2\}$ and $A_1\vee A_2:=\max\{A_1,A_2\}$.

\subsection{Main results}
We first use a scaling argument to unify several existing works on regularisation by noise (see \cite[Section 1.5]{BFGM} for the Brownian case with $L^q_tL^p_x$ drift), and then present our main results.

Assume that $\{L_t\}_{t\in [0,T]}$ is a symmetric rotationally invariant $\alpha$-stable process with $\alpha \in (0,2]$, with $\alpha=2$ corresponding to the standard Brownian motion. By the scaling property,
\begin{equation}\label{1.7}
Law(L_t)=\kappa^{-\frac{1}{\alpha}}Law(L_{\kappa t}), \ \ \forall \ \kappa>0.
\end{equation}
To ensure that the regularizing effects of $L_\cdot$ dominate the singularities of $b$ in \eqref{1.1}, it is natural to require that, under rescaling which preserves the noise strength, the nonlinearity vanishes; otherwise, the nonlinearity would dominate and may lead to pathologies such as coalescence or branching of solutions.

With \eqref{1.7} in
mind, we define the rescaled drift for $\vartheta\in (0,1)$ by
\begin{equation}\label{1.8}
b^\vartheta(t,x)=\vartheta^{1-\frac{1}{\alpha}}b(\vartheta t,\vartheta^{\frac{1}{\alpha}}x).
\end{equation}
Let $\mathbb{V}$ be a function space on ${\mathbb R}_+\times{\mathbb R}^d$. The leading-order seminorm of $b^\vartheta$, denoted by $[b^\vartheta ]_{\mathbb{V}}$, typically scales as $\vartheta^\delta[b]_{{\mathbb V}}$. We call ${\mathbb V}$ critical if $\delta=0$, subcritical if $\delta>0$, and supercritical if $\delta<0$.

$\bullet$ If $\alpha=2$ and ${\mathbb V}=L^q({\mathbb R}_+;L^p({\mathbb R}^d;{\mathbb R}^d))$ with $p,q\in [2,\infty]$, then
\[
[b^\vartheta]_{\mathbb{V}}=\|b^\vartheta\|_{L^q({\mathbb R}_+;L^p({\mathbb R}^d))}=
\vartheta^{\frac{1}{2}-\frac{d}{2p}-\frac{1}{q}}\|b\|_{L^q({\mathbb R}_+;L^p({\mathbb R}^d))}=
\vartheta^{\frac{1}{2}-\frac{d}{2p}-\frac{1}{q}}[b]_{{\mathbb V}}.
\]
Thus,
\begin{equation}\label{1.9}
\left\{
\begin{array}{ll}
 {\rm Subcriticality}: & \frac{d}{p}+\frac{2}{q}<1,  \\ [0.2cm]
{\rm Criticality}: & \frac{d}{p}+\frac{2}{q}=1,  \\ [0.2cm]
{\rm Supercriticality}: &  \frac{d}{p}+\frac{2}{q}>1. \end{array}
\right.
\end{equation}
This coincides with the classical Ladyzhenskaya--Prodi--Serrin condition for the Navier--Stokes equations (the subcritical case is also known as the Krylov--R\"{o}ckner condition). Strong well-posedness has been established in the subcritical case by Krylov and R\"{o}ckner \cite{KR}, and in the critical case by R\"{o}ckner and Zhao \cite{RZ}. In the supercritical case, the nonexistence of strong solution was shown by Krylov \cite{Kry21-1} for time-independent drifts, while weak existence was proved by Butkovsky and Gallay \cite{BG} under the condition $d/p+1/q<1$.

$\bullet$ If $\alpha=2$ and ${\mathbb V}=L^p({\mathbb R}_+;{\mathcal C}_b^\beta({\mathbb R}^d;{\mathbb R}^d))$ with $p\in (1,2)$, $\beta\in (0,1)$, then
\[
[b^\vartheta]_{\mathbb{V}}=[b^\vartheta]_{p,\beta}=
\vartheta^{\frac{1}{2}-\frac{1}{p}+\frac{\beta}{2}}[b]_{p,\beta}=
\vartheta^{\frac{1}{2}-\frac{1}{p}+\frac{\beta}{2}}[b]_{\mathbb{V}}.
\]
It follows that
\begin{equation}\label{1.10}
\left\{
\begin{array}{ll}
 {\rm Subcriticality}: & p>\frac{2}{1+\beta},  \\ [0.2cm]
{\rm Criticality}: & p=\frac{2}{1+\beta},  \\ [0.2cm]
{\rm Supercriticality}: & p<\frac{2}{1+\beta}. \end{array}
\right.
\end{equation}
Strong well-posedness in the subcritical case and ill-posedness in the supercritical case were proved by Galeati and Gerencs\'{e}r \cite{GG} via the stochastic sewing lemma, while an alternative proof of strong well-posedness in the subcritical case was obtained by Wei, Hu and Yuan \cite{WHY1} using It\^{o}--Tanaka's trick. 

$\bullet$ If $\alpha\in (0,2)$ and ${\mathbb V}=L^\infty({\mathbb R}_+;\mathcal{B}^\beta_{p,\infty}({\mathbb R}^d;{\mathbb R}^d))$ with $p\in (1,\infty)$ and $\beta\in {\mathbb R}$, then
\begin{equation}\label{1.11}
[b^\vartheta]_{\mathbb{V}}=\|b^\vartheta\|_{L^\infty({\mathbb R}_+;\dot{\mathcal{B}}^\beta_{p,\infty}({\mathbb R}^d))}=
\vartheta^{1-\frac{1}{\alpha}-\frac{d}{\alpha p}+\frac{\beta}{\alpha}}\|b\|_{L^\infty({\mathbb R}_+;\dot{\mathcal{B}}^\beta_{p,\infty}({\mathbb R}^d))}=
\vartheta^{1-\frac{1}{\alpha}-\frac{d}{\alpha p}+\frac{\beta}{\alpha}}[b^\vartheta]_{\mathbb{V}},
\end{equation}
where  $\mathcal{B}^\beta_{p,\infty}({\mathbb R}^d)$ and $\dot{\mathcal{B}}^\beta_{p,\infty}$ denote the nonhomogeneous and homogeneous Besov spaces, respectively. Therefore,
\begin{equation}\label{1.12}
\left\{
\begin{array}{ll}
 {\rm Subcriticality}: & p>\frac{d}{\alpha+\beta-1} \ \ {\rm and} \ \ \alpha+\beta>1,  \\ [0.2cm]
{\rm Criticality}: & p=\frac{d}{\alpha+\beta-1}  \ \ {\rm and} \ \  \alpha+\beta>1,  \\ [0.2cm]
{\rm Supercriticality}: &  p<\frac{d}{\alpha+\beta-1} \ \ {\rm and} \ \ \alpha+\beta>1. \end{array}
\right.
\end{equation}
Weak well-posedness in the subcritical case was established by Song and Xie \cite{SX}. Pathwise uniqueness was further obtained under the additional assumptions $\beta>1-\alpha/2$ and $p>2d/\alpha$. However, the strong well-posedness and ill-posedness in the critical and supercritical cases remain open.

$\bullet$ If $\alpha\in (0,2)$ and ${\mathbb V}=L^p({\mathbb R}_+;{\mathcal C}_b^\beta({\mathbb R}^d;{\mathbb R}^d))$ with $p\in (1,\infty)$, $\beta\in (0,1)$, then
\[
[b^\vartheta]_{\mathbb{V}}=[b^\vartheta]_{p,\beta}=
\vartheta^{1-\frac{1}{\alpha}-\frac{1}{p}+\frac{\beta}{\alpha}}[b]_{p,\beta}=
\vartheta^{1-\frac{1}{\alpha}-\frac{1}{p}+\frac{\beta}{\alpha}}[b]_{\mathbb{V}}.
\]
Hence,
\begin{equation}\label{1.13}
\left\{
\begin{array}{ll}
 {\rm Subcriticality}: & p>\frac{\alpha}{\alpha+\beta-1} \ \ {\rm and} \ \ \alpha+\beta>1, \\ [0.2cm]
{\rm Criticality}: & p=\frac{\alpha}{\alpha+\beta-1}  \ \ {\rm and} \ \  \alpha+\beta>1,  \\ [0.2cm]
{\rm Supercriticality}: &  p<\frac{\alpha}{\alpha+\beta-1} \ \ {\rm and} \ \ \alpha+\beta>1. \end{array}
\right.
\end{equation}
The strong well-posedness and ill-posedness of \eqref{1.1} with the drift in the fully subcritical, critical and supercritical regimes remain unclear.

In this paper, we  use It\^{o}--Tanaka's trick to establish strong well-posedness in the subcritical regime and the comparison principle to prove the strong ill-posedness in the supercritical regime for \eqref{1.1}. Our first main result is given as the following.

\begin{theorem}\label{th1.1} Suppose $d\geq1$ and $b\in L^1([0,T];{\mathcal C}_{bu}({\mathbb R}^d;{\mathbb R}^d))$. Let $\{L_t\}_{t\in [0,T]}$ be a $d$-dimensional symmetric rotationally invariant $\alpha$-stable process with $\alpha \in (0,2)$ on a filtered probability space $(\Omega,{\mathcal F},$ $\{{\mathcal F}_t\}_{t\in[0,T]},{\mathbb P})$ starting from 0.

\smallskip
(i) There exists a weak solution to \eqref{1.1}, i.e. there is a filtered probability space $(\tilde{\Omega},\tilde{\mathcal{F}},$ $\{\tilde{\mathcal{F}}_t\}_{t\in [0,T]},\tilde{\mathbb{P}})$ with processes  $\{\tilde{L}_t\}_{t\in [0,T]}$ and $\{\tilde{X}_t\}_{t\in [0,T]}$ defined on it such that $\{\tilde{L}_t\}_{t\in [0,T]}$ is a $d$-dimensional
symmetric rotationally invariant $\alpha$-stable process  with $\alpha \in (0,2)$ and $\tilde{L}_0=0$, and $\{\tilde{X}_t\}_{t\in [0,T]}$ is
$\{\tilde{\mathcal{F}}_t\}_{t\in [0,T]}$-adapted and c\'{a}dl\`{a}g,
satisfying for all $t\in [0,T]$,
\begin{equation}\label{1.14}
\tilde{X}_t=x+\int_0^tb(s,\tilde{X}_s)ds+\tilde{L}_t, \ \ {\mathbb P}-a.s..
\end{equation}

(ii) If $b\in L^p([0,T];{\mathcal C}_b^{\beta}({\mathbb R}^d;{\mathbb R}^d))$ with $\beta\in (0,1)$ and $p\in (1,\infty]$ such that
\begin{equation}\label{1.15}
\alpha+\beta>1 \ \ {\rm and} \ \ p>\frac{\alpha}{\alpha+\beta-1},
\end{equation}
then the uniqueness in law holds.

\smallskip
(iii) If the condition in (ii) is satisfied and, in addition, $\beta>1-\alpha/2$, then pathwise uniqueness holds.

\smallskip
(iv) Under the assumptions in (iii), Davie's type uniqueness also holds. That is, there exists an event $\Omega^\prime\in {\mathcal F}$ with ${\mathbb P}(\Omega^\prime)=1$ such that for any $\omega\in\Omega^\prime, \, x\in{\mathbb R}^d$, the integral equation
\begin{equation}\label{1.16}
\varphi(t)=x+\int_0^tb(s,\varphi(s)+L_s(\omega))ds, \quad t\in [0,T],
\end{equation}
has exactly one solution $\varphi$ in ${\mathcal C}([0,T];{\mathbb R}^d)$.
\end{theorem}

\begin{remark} \label{rem1.2} (i) When $d=1$, uniqueness in law implies pathwise uniqueness. Thus, under the subcritical condition and for $d=1$, \eqref{1.1} is uniquely solvable in strong sense.

\smallskip
(ii) When $b$ is bounded and $\{L_t\}_{t\in [0,T]}$ is a $d$-dimensional standard Brownian motion,  Davie's type uniqueness (also called path-by-path uniqueness) was first established by Davie \cite{Dav07}. Priola later extended this result from Brownian motion to symmetric rotationally invariant $\alpha$-stable processes with $\alpha\in(0,2)$, proving Davie's type uniqueness for \eqref{1.1} with $b\in L^\infty([0,T];\mathcal C_b^\beta(\mathbb R^d;\mathbb R^d))$ and $\beta>1-\alpha/2$ (see \cite[Theorem 1.1]{Priola18}). More recently, Tian and Wei obtained the same result as Priola for $\alpha\in[1,2)$ under the weaker assumption $p>\max\{\alpha/(\alpha-1),2\alpha/(\alpha+2\beta-2)\}$ together with $\beta>1-\alpha/2$ (see \cite[Theorem 2]{TW}). In Theorem \ref{th1.1} (iv), however, the integrability index $p$ is only required to lie in the subcritical regime, thereby extending the existing results.
\end{remark}

The above strong  well-posedness for \eqref{1.1} is sharp, in the sense that if the drift $b$ belongs to the supercritical Lebesgue--H\"{o}lder space, pathwise uniqueness fails even when $\alpha\in (0,2)$, $\beta\in (0,1)$ and $\beta>1-\alpha/2$. The nonuniqueness is characterized by the following theorem.

\begin{theorem}\label{th1.3} Let $\{L_t\}_{t\in [0,T]}$ be as in Theorem \ref{th1.1}, and let $b\in L^p([0,T];{\mathcal C}_b^{\beta}({\mathbb R}^d;{\mathbb R}^d))$ with $d\geq 1$, $\beta\in (0\vee (1-\alpha),1)$ and $p\in [1,\infty)$. If $p<\alpha/(\alpha+\beta-1)$, then there exist an initial condition $x\in {\mathbb R}^d$ and a drift $b$ such that there exist two distinct solutions to \eqref{1.1}.
\end{theorem}

\section{Fractional Kolmogorov equations}\label{sec2}\setcounter{equation}{0}
In this section, let us discuss the following fractional Kolmogorov equation
\begin{equation}\label{2.1}
\left\{
\begin{array}{ll}
 \partial_t u(t,x)=\Delta^{\frac{\alpha}{2}} u(t,x)+b(t,x)\cdot \nabla u(t,x)-\lambda u(t,x)\\  [0.1cm] \qquad\qquad\quad +f(t,x), \ \ (t,x)\in(0,T] \times {\mathbb R}^d, \\ [0.1cm]
 u(t,x)|_{t=0}=0, \ \ x\in{\mathbb R}^d,
 \end{array}
\right.
\end{equation}
where the final time horizon $T>0$, the dimension $d\geq 1$, the parameter $\lambda\geq 0$  and the index $\alpha\in (0,2)$. The function $f:[0,T]\times {\mathbb R}^d\rightarrow {\mathbb R}$ and $b:[0,T]\times {\mathbb R}^d
\rightarrow {\mathbb R}^d$ are assumed to be Borel measurable. If $u\in
L^1([0,T];{\mathcal C}^{1\vee \alpha}({\mathbb R}^d))\cap W^{1,1}([0,T];{\mathcal C}({\mathbb R}^d))$ satisfying $u(t,x)|_{t=0}=0$ such that the first equation of \eqref{2.1} holds true for almost all $(t,x)\in (0,T)\times {\mathbb R}^d$,
then the unknown function $u$ is said to be a strong solution. Firstly, let us establish the well-posedness for \eqref{2.1}.

\begin{lemma} \label{lem2.1} Let $p\in (1,\infty]$, $\alpha\in (0,2)$, $\beta\in (0,1)$ and $\lambda\geq0$. Suppose that
\[
f\in L^p([0,T];{\mathcal C}^\beta_b({\mathbb R}^d)) \ \ {\rm and} \ \ b\in L^\infty([0,T];{\mathcal C}^\beta_b({\mathbb R}^d;{\mathbb R}^d))
\]
with $1<\alpha+\beta\neq 2$. Then there is a unique strong solution $u$ to \eqref{2.1}.

(i) If $\alpha\in [1,2)$, then
\[
u\in L^p([0,T];{\mathcal C}_b^{\alpha+\beta}({\mathbb R}^d))\cap W^{1,p}([0,T];{\mathcal C}_b^\beta({\mathbb R}^d)).
\]

(ii) If $\alpha\in(0,1)$, then for every $\theta\in(0,\alpha+\beta-1)$ we have
\[
u\in L^p([0,T];{\mathcal C}_b^{\alpha+\beta-\theta}({\mathbb R}^d))\cap W^{1,p}([0,T];{\mathcal C}_b^{\alpha+\beta-1-\theta}({\mathbb R}^d)).
 \]
Moreover, there exists a positive constant $C=C(d,T,\alpha,\beta,p,\lambda,\theta,\|b\|_{\infty,\beta})$ such that
\begin{equation}\label{2.2}
\|u\|_{p,\alpha+\beta-\theta}\leq C\|f\|_{p,\beta}.
\end{equation}
\end{lemma}
\begin{proof} Suppose that $u$ solves \eqref{2.1}. For any $\bar{\lambda}\in {\mathbb R}$, define $\bar{u}(t,x)=u(t,x)e^{-(\bar{\lambda}-\lambda)t}$. Then $\bar{u}$ solves the following Cauchy problem
\begin{equation*}
\left\{\begin{array}{ll}
\partial_{t}\bar{u}(t,x)=\Delta^{\frac{\alpha}{2}}\bar{u}(t,x)+b(t,x)\cdot\nabla \bar{u}(t,x) -\bar{\lambda} \bar{u}(t,x)\\ [0.1cm] \qquad\qquad\quad +\bar{f}(t,x), \ \ (t,x)\in (0,T]\times{\mathbb R}^d , \\ [0.1cm]
\bar{u}(t,x)|_{t=0}=0, \  x\in{\mathbb R}^d,
\end{array}\right.
\end{equation*}
where $\bar{f}(t,x)=f(t,x)e^{-(\bar{\lambda}-\lambda)t}$. Conversely, if $\bar{u}$ solves the above problem, then $u(t,x)=\bar{u}(t,x)e^{(\bar{\lambda}-\lambda)t}$ solves \eqref{2.1}.
Thus, it suffices to prove conclusions for sufficiently large $\lambda$.

\smallskip
(i) When the drift vanishes, by \cite[Lemma 3 and Remark 3]{TW}, there exists a unique solution
\[
u\in L^p([0,T];{\mathcal C}_b^{\alpha+\beta}({\mathbb R}^d))\cap W^{1,p}([0,T];{\mathcal C}_b^\beta({\mathbb R}^d))
\]
to the Cauchy problem \eqref{2.1}, which satisfies
\begin{equation}\label{2.3}
\|u\|_{p,\alpha+\beta}\leq C\|f\|_{p,\beta}.
\end{equation}
For general $b\in L^\infty([0,T];{\mathcal C}^\beta_b({\mathbb R}^d;{\mathbb R}^d))$, we suppose that  $u\in L^p([0,T];{\mathcal C}_b^{\alpha+\beta}({\mathbb R}^d))$ is a solution of \eqref{2.1}, then \eqref{2.3} implies
\[
\|u\|_{p,\alpha+\beta}\leq C\big[\|b\cdot \nabla u\|_{p,\beta}+\|f\|_{p,\beta}\big]\leq C\big[\|b\|_{\infty,\beta}\|\nabla u\|_{p,\beta}+\|f\|_{p,\beta}\big].
\]
On the other hand, the unique solution has the following heat-kernel representation (see \cite[Lemma 2.1]{TDW})
\begin{equation*}
u(t,x)=\int_0^te^{-\lambda (t-s)}ds\int_{{\mathbb R}^d}K(t-s,x-y)[b(s,y)\cdot \nabla u(s,y)+f(s,y)]dy.
\end{equation*}
Since $\alpha\in [1,2)$, for large enough $\lambda$, one concludes
\[
\|\nabla u\|_{p,\beta}\leq C\|f\|_{p,\beta}\big].
\]
Using the continuity method, we get the conclusion.

\smallskip
(ii) For $\alpha\in (0,1)$, it suffices to establish the a priori estimate \eqref{2.2} for smooth $u$.

Let $x_0\in {\mathbb R}^d$ and $\tilde{x}_t$ be a solution of the following ordinary differential equation (ODE)
\[
\dot{\tilde{x}}_t=-b(t,x_0+\tilde{x}_t), \quad \tilde{x}_t|_{t=0}=0.
\]
Define
\begin{equation*}
\left\{\begin{array}{ll}
\tilde{u}(t,x):=u(t,x+x_0+\tilde{x}_t), \ \tilde{b}(t,x):=b(t,x+x_0+\tilde{x}_t)-b(t,x_0+\tilde{x}_t), \\ [0.1cm] \tilde{f}(t,x):=f(t,x+x_0+\tilde{x}_t).  \end{array}\right.
\end{equation*}
Then $\tilde{u}$ satisfies
\begin{equation*}
\left\{\begin{array}{ll}
\partial_{t}\tilde{u}(t,x)=\Delta^{\frac{\alpha}{2}}\tilde{u}(t,x)+\tilde{b}(t,x)\cdot \nabla \tilde{u}(t,x) -\lambda \tilde{u}(t,x)\\ [0.1cm] \qquad\qquad\quad  +\tilde{f}(t,x), \ (t,x)\in (0,T]\times {\mathbb R}^d, \\ [0.1cm]
\tilde{u}(t,x)|_{t=0}=0, \  x\in{\mathbb R}^d.  \end{array}\right.
\end{equation*}
Moreover, $\tilde{u}$ has the following equivalent integral representation
\begin{equation}\label{2.4}
\tilde{u}(t,x)=\int_0^te^{-\lambda (t-s)}ds\int_{{\mathbb R}^d}K(t-s,x-y)[\tilde{b}(s,y)\cdot \nabla \tilde{u}(s,y)+\tilde{f}(s,y)]dy,
\end{equation}
where $K(t,x)$ is the heat kernel associated with the fractional Laplacian $\Delta^{\frac{\alpha}{2}}$.

For each $\theta\in [0, \beta]$, we have
\begin{equation}\label{2.5}
\left\{\begin{array}{ll}
|\tilde{b}(s,y)\cdot \nabla \tilde{u}(s,y)|\leq [b]_{\infty,\beta-\theta} |y|^{\beta-\theta} \|\nabla u(s,\cdot)\|_0, \\ [0.2cm]
|\tilde{f}(s,y)-\tilde{f}(s,0)|\leq [f(s,\cdot)]_{\beta-\theta} |y|^{\beta-\theta}.
\end{array}\right.
\end{equation}
Differentiating both sides of \eqref{2.4} with respect to $x$ and then taking the supremum in $x_0\in\mathbb R^d$ yield
\begin{equation}\label{2.6}
\begin{split}\|\nabla u(t,\cdot)\|_0 \leq& \;C(d,\alpha)\int_0^te^{-\lambda (t-s)}ds\int_{{\mathbb R}^d}
\frac{(t-s)|y|^\beta}{[(t-s)^{\frac{1}{\alpha}}+|y|]^{d+1+\alpha}}
\Big[[b]_{\infty,\beta} \|\nabla u(s,\cdot)\|_0+[f(s,\cdot)]_\beta\Big] dy\\ \leq &\;C(d,\alpha,\beta)\int_0^te^{-\lambda (t-s)}(t-s)^{\frac{\beta-1}{\alpha}}
\Big[[b]_{\infty,\beta} \|\nabla u(s,\cdot)\|_0+[f(s,\cdot)]_\beta\Big] ds,
\end{split}
\end{equation}
where the first inequality follows from \eqref{1.5} and \eqref{2.5}.

The same procedure used for every $0\leq \theta<\alpha+\beta-1<\beta$, also implies
\begin{equation}\label{2.7}
\begin{split}&\|\nabla u(t,\cdot)\|_0\\ \leq&\; C(d,\alpha)\int_0^te^{-\lambda (t-s)}ds\int_{{\mathbb R}^d}
\frac{(t-s)|y|^{\beta-\theta}}{[(t-s)^{\frac{1}{\alpha}}+|y|]^{d+1+\alpha}}
\Big[[b]_{\infty,\beta-\theta} \|\nabla u(s,\cdot)\|_0+[f(s,\cdot)]_{\beta-\theta} \Big] dy\\ \leq&\;C(d,\alpha,\beta,\theta)\int_0^te^{-\lambda (t-s)}(t-s)^{\frac{\beta-\theta-1}{\alpha}}
\Big[[b]_{\infty,\beta-\theta} \|\nabla u(s,\cdot)\|_0+[f(s,\cdot)]_{\beta-\theta}\Big] ds.
\end{split}
\end{equation}

By \eqref{2.6} and Young's inequality,  we conclude
\begin{equation}\label{2.8}
\|\nabla u\|_{p,0}\leq C(d,\alpha,\beta,p)
\Big[[b]_{\infty,\beta}\|\nabla u\|_{p,0}+
[f]_{p,\beta}\Big]\int_0^Te^{-\lambda s}s^{\frac{\beta-1}{\alpha}}ds.
\end{equation}
The parameter $\lambda$ is chosen to be large enough such that
\begin{equation}\label{2.9}
C(d,\alpha,\beta,p)[b]_{\infty,\beta}\int_0^Te^{-\lambda s}s^{\frac{\beta-1}{\alpha}}ds<\frac{1}{2}.
\end{equation}
By \eqref{2.8} and \eqref{2.9}, it gives rise to
\begin{equation}\label{2.10}
\|\nabla u\|_{p,0}\leq C(d,T,\alpha,p,\lambda,[b]_{\infty,\beta})
[f]_{p,\beta},
\end{equation}
which also suggests
\begin{equation}\label{2.11}
\|u\|_{p,0}\leq
 C(d,T,\alpha,\beta,p,\lambda,\|b\|_{\infty,\beta})\|f\|_{p,\beta},
\end{equation}
if ones uses \eqref{2.4}.

Let $P_\xi$ be given by \eqref{1.2} and set $v_\xi(t,x)=\partial_\xi P_\xi u(t,x)$. Then
\begin{equation}\label{2.12}
\left\{\begin{array}{ll}
\partial_{t}v_\xi(t,x)=\Delta^{\frac{\alpha}{2}} v_\xi(t,x)+b(t,x)\cdot\nabla v_\xi(t,x)
-\lambda v_\xi(t,x)\\ [0.1cm] \qquad\qquad\quad \ \ +g_\xi(t,x), \ \ (t,x)\in (0,T]\times{\mathbb R}^d, \\ [0.1cm]
v_\xi(t,x)|_{t=0}=0, \  x\in{\mathbb R}^d,
\end{array}\right.
\end{equation}
where
\[
g_\xi(t,x)=\partial_\xi P_\xi f(t,x)+\partial_\xi P_\xi(b(t,x)\cdot \nabla u(t,x))-b(t,x)\cdot \partial_\xi P_\xi \nabla u(t,x).
\]
With the aid of \eqref{2.7}, for every fixed $\theta\in [0,\alpha+\beta-1)$, we achieve
\begin{equation}\label{2.13}
\begin{split}
\|\nabla v_\xi(t,\cdot)\|_0 \leq&\; C\int_0^te^{-\lambda (t-s)}(t-s)^{\frac{\beta-\theta-1}{\alpha}}
\Big[[b]_{\infty,\beta-\theta}\|\nabla v_\xi(s,\cdot)\|_0+[g_\xi(s,\cdot)]_{\beta-\theta} \Big] ds
\\ \leq&\;
C\int_0^te^{-\lambda (t-s)}(t-s)^{\frac{\beta-\theta-1}{\alpha}}
\Big[[b]_{\infty,\beta-\theta}\|\nabla v_\xi(s,\cdot)\|_0+[\partial_\xi P_\xi f(s,\cdot)]_{\beta-\theta} \\ &+[\partial_\xi P_\xi(b(s,\cdot)\cdot \nabla u(s,\cdot))-b(s,\cdot)\cdot  \partial_\xi P_\xi \nabla u(s,\cdot)]_{\beta-\theta}   \Big] ds.
\end{split}
\end{equation}
By \cite[Lemma 2.1]{CSZ}, for every $0<\beta_1\leq \beta_2<1$, there exists a positive constant $C(d,\beta_1,\beta_2)$ such that
\[
[\partial_\xi P_\xi (h_1h_2)-h_1\partial_\xi P_\xi h_2]_{\beta_2-\beta_1}\leq C(d,\beta_1,\beta_2)[h_1]_{\beta_2}\|h_2\|_0\xi^{\beta_1-1},
\]
if $[h_1]_{\beta_2}$ and $\|h_2\|_0$ are finite. This also implies
\begin{equation}\label{2.14}
[\partial_\xi P_\xi f(s,\cdot)]_{\beta-\theta}\leq C(d,\beta,\theta)[f(s,\cdot)]_\beta\xi^{\theta-1}
\end{equation}
and
\begin{equation}\label{2.15}
[\partial_\xi P_\xi(b(s,\cdot)\cdot \nabla u(s,\cdot))-b(s,\cdot)\cdot  \partial_\xi P_\xi \nabla u(s,\cdot)]_{\beta-\theta}\leq C[b]_{\infty,\beta}\|\nabla u(s,\cdot)\|_0\xi^{\theta-1}.
\end{equation}
Taking into account \eqref{2.13}--\eqref{2.15}, we deduce
\begin{equation}\label{2.16}
\begin{split}\sup_{\xi>0}\| \xi^{1-\theta}\partial_\xi P_\xi \nabla u(t,\cdot) \|_0 \leq&\;
C\int_0^te^{-\lambda (t-s)}(t-s)^{\frac{\beta-\theta-1}{\alpha}}
\Big[[b]_{\infty,\beta-\theta}\sup_{\xi>0}\| \xi^{1-\theta}\partial_\xi P_\xi \nabla u(s,\cdot) \|_0\\ &+[b]_{\infty,\beta}\|\nabla u(s,\cdot)\|_0+[f(s,\cdot)]_\beta  \Big] ds.
\end{split}
\end{equation}
We use Young's inequality again to \eqref{2.16}, and get
\begin{equation*}
\begin{split}
&\Big\|\sup_{\xi>0}\| \xi^{1-\theta}\partial_\xi P_\xi \nabla u(t,\cdot) \|_0\Big\|_{L^p([0,T])}
\\ \leq&\; C(d,T,\alpha,\beta,p,\theta)\|b\|_{\infty,\beta}\int_0^Te^{-\lambda s} s ^{\frac{(\beta-\theta-1)}{\alpha}}ds\Big\|\sup_{\xi>0}\| \xi^{1-\theta}\partial_\xi P_\xi \nabla u(t,\cdot) \|_0\Big\|_{L^p([0,T])}
\\ &+C(d,T,\alpha,\beta,p,\theta)\Big[\|b\|_{\infty,\beta}
\|\nabla u\|_{p,0}+[f]_{p,\beta}\Big]
\int_0^Te^{-\lambda s} s^{\frac{(\beta-\theta-1)}{\alpha}}ds,
\end{split}
\end{equation*}
where $\theta\in (0,\alpha+\beta-1)$.

By taking $\lambda$ sufficiently big such that
\[
C(d,T,\alpha,\beta,p,\theta)\|b\|_{\infty,\beta}\int_0^Te^{-\lambda s}s^{\frac{(\beta-\theta-1)}{\alpha}}ds<\frac{1}{2},
\]
it yields to
\begin{equation}\label{2.17}
\Big\|\sup_{\xi>0}\| \xi^{1-\theta}\partial_\xi P_\xi \nabla u(t,\cdot) \|_0\Big\|_{L^p([0,T])}
\leq C(d,T,\alpha,\beta,p,\lambda,\theta,\|b\|_{\infty,\beta})
\|f\|_{p,\beta}.
\end{equation}
In conjunction with \eqref{1.3}, \eqref{1.4}, \eqref{2.10}, \eqref{2.11} and \eqref{2.17}, we conclude \eqref{2.2}. \end{proof}

\begin{remark} \label{rem2.2} (i) When $\alpha=2$, the maximal Lebesgue--Schauder estimates for \eqref{2.1} were first proved by Krylov \cite{Krylov02}. Here, we extend Krylov's result to the fractional Laplacian operator. Moreover, the estimate \eqref{2.2} is expected to hold for $\theta=0$ if one employs Calder\'{o}n--Zygmund type estimates; however, this issue is not pursued further in the present article.

(ii) When $b$ is locally H\"{o}lder continuous in $x$ and bounded in $t$, and when $f$ is bounded and H\"{o}lder continuous in $x$, the Schauder estimate \eqref{2.3} for $p=\infty$ was obtained by De Raynal, Menozzi and Priola \cite{DMP2}. We extend their result from time-bounded coefficients to time-integrable ones.
\end{remark}

Before establishing the well-posedness of the Cauchy prpblem \eqref{2.1}, we first introduce another useful lemma.
\begin{lemma}(\cite[Theorem 1, p.119]{Stein})\label{lem2.3} (Hardy--Littlewood--Sobolev convolution inequality) Let $\tilde{k}\in \mathbb{N}$, $1<p_1<\infty$ and define $\psi(y)=|y|^{-\tilde{k}/p_1}$. Let $1<p_2<p_3<\infty$ such that $1/p_1+1/p_2=1+1/p_3$. If $h\in L^{p_2}({\mathbb R}^{\tilde{k}})$, then $h\ast\psi\in L^{p_3}({\mathbb R}^{\tilde{k}})$ and there exists a positive constant $C(p_1,p_2)$ such that
\[
\|h\ast\psi\|_{L^{p_3}({\mathbb R}^{\tilde{k}})}\leq C(p_1,p_2)\|h\|_{L^{p_2}({\mathbb R}^{\tilde{k}})}.
\]
\end{lemma}

Now, let us give our main result for \eqref{2.1}.
\begin{theorem} \label{th2.4} \textbf{(Existence and uniqueness)} Let $p\in (1,\infty]$, $\alpha\in (0,2)$, $\beta\in (0,1)$ and $\lambda\geq0$ such that
\begin{equation}\label{2.18}
\alpha+\beta>1 \ \ {\rm and} \ \ p>\frac{\alpha}{\alpha+\beta-1}.
\end{equation}
Suppose $b\in L^p([0,T];{\mathcal C}_b^{\beta}({\mathbb R}^d;{\mathbb R}^d))$ and $f\in L^p([0,T];{\mathcal C}_b^{\beta}({\mathbb R}^d))$. Then there exists a unique strong solution $u$ to the Cauchy problem \eqref{2.1}.

\smallskip
(i) \textbf{(Regularity)} If $\alpha-1-\alpha/p\leq 0$, the unique strong solution further belongs to
\begin{equation}\label{2.19}
\begin{split} {\mathcal H}=&\bigcap_{0\leq \theta <\alpha+\beta-1-\alpha/p}L^\infty([0,T];{\mathcal C}_b^{1+\theta}({\mathbb R}^d))
\\ &\qquad
 \bigcap \bigcap_{\alpha+\beta-1-\alpha/p<\theta<\alpha+\beta-1}
L^{\frac{p\alpha}{\alpha-p(\alpha+\beta-1-\theta)}}([0,T];{\mathcal C}_b^{1+\theta}({\mathbb R}^d)).
\end{split}
\end{equation}
Moreover, there exist positive constants $\varepsilon=\varepsilon(p,\alpha,\beta)$ and $C=C(d,\alpha,\beta,p,[b]_{p,\beta})$ such that, for every large enough $\lambda$,
\begin{equation}\label{2.20}
\sup_{(t,x)\in [0,T]\times {\mathbb R}^d}|\nabla u(t,x)| \leq C\lambda^{-\varepsilon}\|f\|_{p,\beta}.
\end{equation}

(ii) \textbf{(Regularity)} Otherwise, the unique strong solution satisfies \eqref{2.20} and belongs to
\begin{equation}\label{2.21}
\left\{
\begin{array}{ll}
L^p([0,T];{\mathcal C}_b^{\alpha+\beta}({\mathbb R}^d))\bigcap L^\infty([0,T];{\mathcal C}_b^{\alpha+\beta-\frac{\alpha}{p}}({\mathbb R}^d)), \ \ {\rm if} \ \alpha+\beta<2 \ {\rm or} \ \alpha+\beta-\frac{\alpha}{p}>2, \\ [0.2cm]  \bigcap\limits_{0<\theta <1} L^p([0,T];{\mathcal C}_b^{2-\theta}({\mathbb R}^d))\bigcap L^\infty([0,T];{\mathcal C}_b^{\alpha+\beta-\frac{\alpha}{p}}({\mathbb R}^d)), \ \   {\rm if}  \ \alpha+\beta=2, \\ [0.3cm]  L^p([0,T];{\mathcal C}_b^{\alpha+\beta}({\mathbb R}^d)) \bigcap \bigcap\limits_{0<\theta <1} L^\infty([0,T];{\mathcal C}_b^{2-\theta}({\mathbb R}^d)), \ \  {\rm if}  \ \alpha+\beta-\frac{\alpha}{p}=2.
 \end{array}
\right.
\end{equation}
\end{theorem}
\begin{proof} Clearly, it suffices to show the conclusions for some large enough $\lambda$. We divide the proof into two cases: $\alpha-1-\alpha/p\leq 0$ and $\alpha-1-\alpha/p>0$.

For the case $\alpha-1-\alpha/p>0$, this implies $\alpha>1$ and $p>\alpha/(\alpha-1)$. If $\alpha+\beta\neq 2$, then by \cite[Theorem 1]{TW}, there is a unique strong solution to \eqref{2.1}. Furthermore,  \eqref{2.20} and \eqref{2.21} hold mutatis mutandis. On the other hand, for every $\beta^\prime\in (0,\beta)$, we have
\[
b\in L^p([0,T];{\mathcal C}_b^{\beta^\prime}({\mathbb R}^d;{\mathbb R}^d)) \ \ {\rm and} \ \ f\in L^p([0,T];{\mathcal C}_b^{\beta^\prime}({\mathbb R}^d))
\]
wheneven $b\in L^p([0,T];{\mathcal C}_b^{\beta}({\mathbb R}^d;{\mathbb R}^d))$ and $f\in L^p([0,T];{\mathcal C}_b^{\beta}({\mathbb R}^d))$. Applying \cite[Theorem 1]{TW} once more, \eqref{2.1} exists a unique strong solution that satisfies \eqref{2.20}. Moreover, if $\alpha+\beta=2$, then
\begin{equation}\label{2.22}
 u\in \bigcap_{0<\theta <\alpha+\beta-1-\alpha/p} \Big[L^p([0,T];{\mathcal C}_b^{2-\theta}({\mathbb R}^d))\cap L^\infty([0,T];{\mathcal C}_b^{\alpha+\beta-\frac{\alpha}{p}-\theta}({\mathbb R}^d))\Big].
\end{equation}
Thus, it remains to verify conclusions for $\alpha-1-\alpha/p\leq 0$ and
\begin{equation}\label{2.23}
u\in L^\infty([0,T];{\mathcal C}_b^{\alpha+\beta-\frac{\alpha}{p}}({\mathbb R}^d)), \ \  {\rm if}  \ \alpha+\beta=2 \ \ {\rm and} \ \ p>\frac{\alpha}{\alpha-1}.
\end{equation}
For \eqref{2.23}, by virtue of \eqref{2.22} and the continuity method, it is sufficient to show
\begin{equation}\label{2.24}
|\nabla u(t,x)-\nabla u(t,y)|\leq C\|f\|_{p,\beta}|x-y|^{\alpha+\beta-1-\frac{\alpha}{p}}, \quad \forall \ |x-y|\leq \frac{1}{3}, \ \ t\in [0,T],
\end{equation}
for regular solution $u$ with $b\equiv0$.

By the integral representation of $u$ (see \eqref{2.4} with $b\equiv0$), for every $x,y\in\mathbb{R}^d$ with $|x-y|\leq 1/3$, the following identity holds
\begin{equation}\label{2.25}
\begin{split}& \nabla u(t,x)-\nabla u(t,y)\\
=&\int_0^te^{-\lambda(t-s)}
ds\int_{|x-z|\leq 2|x-y|}\nabla K(t-s,x-z)[f(s,z)-f(s,x)]dz
\\&-\int_0^t
e^{-\lambda(t-s)}ds\int_{|x-z|\leq 2|x-y|}\nabla K(t-s,y-z)[f(s,z)-f(s,y)]dz
\\&+\int_0^t
e^{-\lambda(t-s)}ds\int_{|x-z|> 2|x-y|}\nabla K(t-s,y-z)[f(s,y)-f(s,x)]dz
\\
&+\int_0^t
e^{-\lambda(t-s)}ds\int_{|x-z|>2|x-y|}[\nabla K(t-s,x-z)-\nabla K(t-s,y-z)][f(s,z)-f(s,x)]dz \\ =&:I_1(t,x,y)+I_2(t,x,y)+I_3(t,x,y)+I_4(t,x,y).
\end{split}
\end{equation}
We first estimate $I_1$ and $I_2$ that
\begin{equation}\label{2.26}
\begin{split}
|I_1(t,x,y)|+|I_2(t,x,y)|\leq&\; C[f]_{p,\beta} \int_{|z|\leq 3|x-y|} |z|^\beta dz \bigg[ \int_0^T \frac{1}{[s^{\frac{1}{\alpha}}+|z|]^{(d+1)p^\prime}}
ds \bigg]^{\frac{1}{p^\prime}}
\\ \leq&\;   C[f]_{p,\beta} |x-y|^{\alpha+\beta-1-\frac{\alpha}{p}},
\end{split}
\end{equation}
where $p^\prime=p/(p-1)$.

Applying Gauss--Green's formula to $I_3$, we arrive at
\begin{equation}
\begin{split}\label{2.27}
|I_3(t,x,y)|=&\; d\bigg|\int_0^tds\int_{|x-z|=2|x-y|}K(t-s,y-z)[f(s,y)-f(s,x)]dS \bigg|
 \\
 \leq&\; \left\{
          \begin{array}{ll}
            C[f]_{p,\beta}|x-y|^{d+\beta-1}\Big[ \int_0^T \frac{1}{[s^{\frac{1}{\alpha}}+|x-y|]^{dp^\prime}}
ds \Big]^{\frac{1}{p^\prime}},  \ {\rm if} \ d\geq 2, \\
             C[f]_{p,\beta}|x-y|^\beta\Big[ \int_0^T
\Big|K(s,y-z)|_{z=x+2|x-y|}^{z=x-2|x-y|}\Big|^{p^\prime}ds\Big]^{\frac{1}{p^\prime}}, \  {\rm if} \ d=1,
          \end{array}
        \right.
 \\ \leq&\;
 \left\{
          \begin{array}{ll}
           C[f]_{p,\beta}|x-y|^{\alpha+\beta-1-\frac{\alpha}{p}}, \ {\rm if} \ d\geq 2, \\
             C|x-y|^{1+\beta}\Big[\int_0^T
|\partial_xK(s,|x-y|)|^{p^\prime}ds\Big]^{\frac{1}{p^\prime}}, \  {\rm if} \ d=1,
          \end{array}
        \right.
 \\
 \leq&\; C[f]_{p,\beta}|x-y|^{\alpha+\beta-1-\frac{\alpha}{p}},
\end{split}
\end{equation}
where $dS$ denote the spherical surface measure and in the sixth line we used the Newton--Leibniz formula.

For $\tau\in [0,1]$, we have $|x-z|/2 \leq |\tau y+(1-\tau)x-z|\leq 2|x-z|$ whenever $|x-z|>2|x-y|$. Hence, by the mean value theorem,
\begin{equation}\label{2.28}
\begin{split}
|I_4(t,x,y)|\leq&\;  C|x-y|[f]_{p,\beta}
\int_{|z|>2|x-y|}|z|^\beta dz \bigg[\int_0^T \frac{1}{[s^{\frac{1}{\alpha}}+|z|]^{(d+2)p^\prime}}ds\bigg]^{\frac{1}{p^\prime}}
 \\ \leq&\; C[f]_{p,\beta}|x-y|^{\alpha+\beta-1-\frac{\alpha}{p}},
\end{split}
\end{equation}

Therefore, combining \eqref{2.25}--\eqref{2.28}, we obtain \eqref{2.24}. It remains to prove the conclusions for the case $\alpha-1-\alpha/p\leq 0$.

\smallskip
$\bullet$ \textbf{(Existence and regularity)} Notice that $b\in L^p([0,T];\mathcal{C}_b^\beta(\mathbb{R}^d;\mathbb{R}^d))$. Then there exists a version of $b$, still denoted by itself, such that $b(0,\cdot)\in \mathcal{C}_b^\beta(\mathbb{R}^d;\mathbb{R}^d)$. We extend $b$ from $[0,T]$ to $(-\infty,T]$ by setting
\[
b(t,x)=b(0,x), \quad {\rm if}  \ t<0.
\]
Let $\varrho$ be a nonnegative normalized mollifier in ${\mathbb R}$,
\[
0\leq \varrho \in {\mathcal C}^\infty_0({\mathbb R}), \ \  {\rm supp}(\varrho)\subset [0,1] \ \ {\rm and} \ \ \int_{{\mathbb R}}\varrho(t)dt=1.
\]
For $n\in{\mathbb N}$, define $\varrho_n(t)=n\varrho(nt)$ and smooth $b$ in time by convolution with  $\varrho_n$,
\[
b_n(t,x)=(b(\cdot,x)\ast \varrho_n)(t)=\int_{{\mathbb R}}b(t-s,x)\varrho_n(s)ds.
\]
Then $b_n\in L^\infty([0,T];{\mathcal C}_b^\beta({\mathbb R}^d;{\mathbb R}^d))$ and
\begin{equation}\label{2.29}
\|b_n\|_{p,0}\leq \|b\|_{p,0} \ \ \mbox{and} \ \  [b_n]_{p,\beta}\leq [b]_{p,\beta}, \quad \forall \ n\in {\mathbb N}.
\end{equation}
Moreover, for every $\beta^\prime\in (0,\beta)$,
\begin{equation}\label{2.30}
\left\{
\begin{array}{ll} \lim\limits_{n\rightarrow\infty}\|b_n-b\|_{p,\beta^\prime}=0, & {\rm if} \  p<\infty,\\ [0.2cm]
\lim\limits_{n\rightarrow\infty}\|b_n-b\|_{p_1,\beta^\prime}=0, \ \ \forall \ p_1\in (2,\infty), & {\rm if} \  p=\infty.
 \end{array}
\right.
\end{equation}

Consider the following Cauchy problem
\begin{equation}\label{2.31}
\left\{
\begin{array}{ll}
 \partial_t u_n(t,x)=\Delta^{\frac{\alpha}{2}} u_n(t,x)+b_n(t,x)\cdot \nabla u_n(t,x)\\ [0.1cm] \qquad\qquad \qquad  -\lambda u_n(t,x)+f(t,x), \ \ (t,x)\in(0,T] \times {\mathbb R}^d, \\ [0.1cm]
 u_n(t,x)|_{t=0}=0, \ \ x\in{\mathbb R}^d.
 \end{array}
\right.
\end{equation}

By Lemma \ref{lem2.1}, there exists a unique strong solution
$u_n$ that belongs to
\[
L^p([0,T];{\mathcal C}_b^{\alpha+\beta-\theta}({\mathbb R}^d))\cap W^{1,p}([0,T];{\mathcal C}_b^{\alpha+\beta-1-\theta}({\mathbb R}^d))
\]
for every $\theta \in (0,\alpha+\beta-1)$. Let $x_0\in {\mathbb R}^d$ and let $x_t^n$ be a solution of the following ODE
\[
\dot{x}_t^n=-b_n(t,x_0+x_t^n), \quad x_t^n|_{t=0}=0.
\]
Define $\hat{u}_n(t,x):=u_n(t,x+x_0+x_t^n)$, then$\hat{u}_n$ satisfies the integral equation \eqref{2.4} with
\[
\hat{b}_n(t,x):=b_n(t,x+x_0+x_t^n)-b_n(t,x_0+x_t^n) \ \ {\rm and} \ \ \hat{f}_n(t,x):=f(t,x+x_0+x_t^n).
\]
For every $0\leq \hat{\theta}\leq \beta$, we obtain the analogue of \eqref{2.5}
\begin{equation*}
\left\{\begin{array}{ll}
|\hat{b}_n(s,y)\cdot \nabla \hat{u}_n(s,y)|\leq [b_n(s,\cdot)]_{\beta-\hat{\theta}} |y|^{\beta-\hat{\theta}} \|\nabla u_n(s,\cdot)\|_0,  \\ [0.2cm] |\hat{f}_n(s,y)-\hat{f}_n(s,0)|\leq [f(s,\cdot)]_{\beta-\hat{\theta}} |y|^{\beta-\hat{\theta}}. \end{array}\right.
\end{equation*}

Furthermore, using the integral representation, we get analogues of \eqref{2.6} and \eqref{2.7} that
\begin{equation}\label{2.32}
\begin{split}\sup_{0\leq t\leq T}\|\nabla u_n(t,\cdot)\|_0\leq&\;
C(d,\alpha,\beta)\int_0^te^{-\lambda (t-s)}(t-s)^{\frac{\beta-1}{\alpha}}
\Big[[b_n(s,\cdot)]_\beta \|\nabla u_n(s,\cdot)\|_0+[f(s,\cdot)]_\beta \Big] ds
\\ \leq&\;
C(d,\alpha,\beta)\Big[[b]_{p,\beta}\sup_{0\leq s\leq T}\|\nabla u_n(s,\cdot)\|_0+[f]_{p,\beta}\Big] \bigg(\int_0^Te^{-\lambda p^\prime s}s^{\frac{(\beta-1)p^\prime}{\alpha}}ds\bigg)^{\frac{1}{p^\prime}}
\\ \leq&\;
C(d,\alpha,\beta)\Big[[b]_{p,\beta}\sup_{0\leq s\leq T}\|\nabla u_n(s,\cdot)\|_0+[f]_{p,\beta}\Big]
 \\ & \quad \times\bigg(\int_0^\infty e^{- p^\prime s}s^{\frac{(\beta-1)p^\prime}{\alpha}}ds\bigg)^{\frac{1}{p^\prime}}
\lambda^{-\frac{p(\alpha+\beta-1)-\alpha}{p\alpha}}
\end{split}
\end{equation}
and
\begin{equation}\label{2.33}
\begin{split}\|\nabla u_n(t,\cdot)\|_0
\leq&\; C(d,\alpha,\beta,\theta)\int_0^te^{-\lambda (t-s)}(t-s)^{\frac{\beta-1-\theta}{\alpha}}
\Big[[b_n(s,\cdot)]_{\beta-\theta} \|\nabla u_n(s,\cdot)\|_0+[f(s,\cdot)]_{\beta-\theta} \Big] ds,
\end{split}
\end{equation}
where $p^\prime=p/(p-1)$, $0\leq \theta<\alpha+\beta-1-\alpha/p\leq \beta$ and in the second inequality of \eqref{2.32} we have used \eqref{2.29} together with H\"{o}lder's inequality.

By condition \eqref{2.18}, i.e., $p>\alpha/(\alpha+\beta-1)$, we get $(\beta-1)p^\prime/\alpha>-1$. This, when combined with \eqref{2.32} and the fact that $\lambda$ is large enough, gives rise to
\begin{equation}\label{2.34}
\sup_{0\leq t\leq T}\|\nabla u_n(t,\cdot)\|_0\leq C(d,\alpha,\beta,p,[b]_{p,\beta})
[f]_{p,\beta}\lambda^{-\frac{p(\alpha+\beta-1)-\alpha}{p\alpha}},
\end{equation}
which further implies
\begin{equation}\label{2.35}
\sup_{0\leq t\leq T}\|u_n(t,\cdot)\|_0\leq
 C(d,\alpha,\beta,p,\lambda,\|b\|_{p,\beta})\|f\|_{p,\beta}.
\end{equation}

The arguments applied to $v_\xi(t,x)$ in \eqref{2.12}--\eqref{2.15}, now adapted to
\[
v_\xi^n(t,x)=\partial_\xi P_\xi u_n(t,x),
 \]
with the only modification that $[b_n(s,\cdot)]_{\beta-\theta}$ replaces $[b]_{\infty,\beta-\theta}$, lead to
\begin{equation}\label{2.36}
\begin{split}&\sup_{\xi>0}\| \xi^{1-\theta}\partial_\xi P_\xi \nabla u_n(t,\cdot) \|_0\\ \leq&\;
C(d,\alpha,\beta,\theta)\int_0^te^{-\lambda (t-s)}(t-s)^{\frac{\beta-\theta-1}{\alpha}}
\Big[[b_n(s,\cdot)]_{\beta-\theta}\sup_{\xi>0}\| \xi^{1-\theta}\partial_\xi P_\xi \nabla u_n(s,\cdot) \|_0\\ &+[b_n(s,\cdot)]_\beta\|\nabla u_n(s,\cdot)\|_0+[f(s,\cdot)]_\beta  \Big] ds.
\end{split}
\end{equation}

Define
\[
\bar{u}_n(t):=\sup_{\xi>0}\| \xi^{1-\theta}\partial_\xi P_\xi \nabla u_n(t,\cdot) \|_0.
\]
Then, by \eqref{2.34}, \eqref{2.36} and H\"{o}lder's inequality, we derive
\begin{equation}\label{2.37}
\begin{split}\|\bar{u}_n\|_{L^\infty([0,T])} \leq&\;
C(d,\alpha,\beta,\theta)\Big[\|\bar{u}_n\|_{L^\infty([0,T])}[b_n]_{p,\beta-\theta}+[b_n]_{p,\beta}\|\nabla u_n\|_{\infty,0}+\|f\|_{p,\beta}\Big]
\\ & \ \times
\bigg[\int_0^Te^{-\lambda s p^\prime} s^{\frac{(\beta-\theta-1)p^\prime}{\alpha}}ds\bigg]^{\frac{1}{p^\prime}} \\ \leq&\;
C(d,\alpha,\beta,p,\theta,\|b\|_{p,\beta})\Big[\|\bar{u}_n\|_{L^\infty([0,T])}\|b\|_{p,\beta}+\|f\|_{p,\beta}\Big]
\\ & \ \times \bigg[\int_0^\infty e^{-s p^\prime} s^{\frac{(\beta-\theta-1)p^\prime}{\alpha}}ds \bigg]^{\frac{1}{p^\prime}}\lambda^{-\frac{\alpha+\beta-1-\alpha/p-\theta}{\alpha}}
 \\ \leq&\;
C(d,\alpha,\beta,p,\theta,\|b\|_{p,\beta})\Big[\|\bar{u}_n\|_{L^\infty([0,T])}
 +\|f\|_{p,\beta}\Big]\lambda^{-\frac{\alpha+\beta-1-\alpha/p-\theta}{\alpha}},
\end{split}
\end{equation}
where $\theta\in [0,\alpha+\beta-1-\alpha/p)$. Since $\lambda$ is big enough, it suggests
\begin{equation}\label{2.38}
\|\bar{u}_n\|_{L^\infty([0,T])}
\leq C(d,\alpha,\beta,p,\theta,\|b\|_{p,\beta})\|f\|_{p,\beta}
\lambda^{-\frac{\alpha+\beta-1-\alpha/p-\theta}{\alpha}}.
\end{equation}
In conjunction with \eqref{1.3}, \eqref{1.4}, \eqref{2.34} and \eqref{2.37}, we have
\begin{equation}\label{2.39}
\|\nabla u_n\|_{\infty,\theta}\leq
C(d,\alpha,\beta,p,\theta,\|b\|_{p,\beta})\|f\|_{p,\beta}\lambda^{-\frac{\alpha+\beta-1-\alpha/p-\theta}{\alpha}}.
\end{equation}

On the other hand, if one defines
\[
\hat{v}_\xi^n(t,x):=\hat{v}_\xi^n(t,x+x_0+x_t^n),
\]
then $\nabla\hat{v}_\xi^n(t,x)$ satisfies the following integral equation
\begin{equation}\label{2.40}
\nabla\hat{v}_\xi^n(t,x)=\int_0^te^{-\lambda (t-s)}ds\int_{{\mathbb R}^d}\nabla K(t-s,x-y)[\hat{b}_n(s,y)\cdot \nabla\hat{v}_\xi^n(s,y)+ \hat{g}_n^\xi(s,y)]dy,
\end{equation}
where
\[
\hat{b}_n(s,y)=b_n(s,y+x_0+x_s^n)-b_n(s,x_0+x_s^n), \ \ \hat{g}_n^\xi(s,y)=g_n^\xi(s,y+x_0+x_s^n),
\]
and
\[
g_n^\xi(s,y)=\partial_\xi P_\xi f(s,y)+\partial_\xi P_\xi(b_n(s,y)\cdot \nabla u_n(s,y))-b_n(s,y)\cdot \partial_\xi P_\xi \nabla u_n(s,y).
\]
Let $\theta\in (0,\alpha+\beta-1-\alpha/p)$ and $\gamma\in (\alpha+\beta-1-\alpha/p-\theta,\alpha+\beta-1-\theta)$. Denote by $\mathscr{I}_\gamma$ the Bessel potential. From \eqref{2.40}, it follows that
\begin{equation}\label{2.41}
\begin{split} & \mathscr{I}_{-\gamma}\nabla \hat{v}_\xi^n(t,x) \\ =&\int_0^te^{-\lambda (t-s)}ds\int_{{\mathbb R}^d}\mathscr{I}_{-\gamma}\nabla K(t-s,x-y)[\hat{b}_n(s,y)\cdot \nabla\hat{v}_\xi^n(s,y)+ \hat{g}_n^\xi(s,y)]dy.
\end{split}
\end{equation}
By \eqref{1.5} and the interpolation inequality, we have
\begin{equation}\label{2.42}
\begin{split}&|\mathscr{I}_{-\gamma}\nabla K(t-s,x-y)|\\ \leq&\; C(d,\alpha,\gamma)(t-s)[(t-s)^\frac{1}{\alpha}+|x-y|)^{-d-\alpha-1}\Big\{1+[(t-s)^\frac{1}{\alpha}+|x-y|]^{-\gamma}\Big\}.
\end{split}
\end{equation}
This, along with \eqref{2.41} and \eqref{2.36}, yields the following
\begin{equation}\label{2.43}
\begin{split}&\sup_{\xi>0}\|\xi^{1-\theta} \partial_\xi P_\xi \mathscr{I}_{-\gamma} \nabla u_n(t,\cdot)\|_0=\sup_{\xi>0}\|\xi^{1-\theta}\mathscr{I}_{-\gamma}\nabla \hat{v}_\xi^n(t,\cdot)\|_0
 \\ \leq&\;C(d,\alpha,\beta,\theta,\gamma)\int_0^t(t-s)^{\frac{\beta-\theta-\gamma-1}{\alpha}}
\Big[[b_n(s,\cdot)]_{\beta-\theta}\sup_{\xi>0}\| \xi^{1-\theta}\partial_\xi P_\xi \nabla u_n(s,\cdot) \|_0\\ &+[b_n(s,\cdot)]_\beta\|\nabla u_n\|_{\infty,0}+[f(s,\cdot)]_\beta  \Big] ds.
\end{split}
\end{equation}

With the aid of Lemma \ref{lem2.3}, together with \eqref{2.34} and \eqref{2.39}, it follows from \eqref{2.43} that
\begin{equation}\label{2.44}
\begin{split} &\bigg[\int_0^T\sup_{(\xi,x)\in {\mathbb R}_+\times{\mathbb R}^d}|\xi^{1-\theta} \partial_\xi P_\xi \mathscr{I}_{-\gamma} \nabla u_n(t,x)|^{\frac{p\alpha}{\alpha-p(\alpha+\beta-1-\tilde{\theta})}}dt\bigg]^{\frac{\alpha-p(\alpha+\beta-1-\tilde{\theta})}{p\alpha}}
 \\  \leq&\; C(d,\alpha,\beta,p,\tilde{\theta},\|b\|_{p,\beta})\|f\|_{p,\beta},
\end{split}
\end{equation}
where $\tilde{\theta}=\theta+\gamma$ and in the above inequality relation \eqref{2.29} have been used. Consequently,
\[
\mathscr{I}_{-\gamma} \nabla u_n\in L^{\frac{p\alpha}{\alpha-p(\alpha+\beta-1-\tilde{\theta})}}([0,T];{\mathcal C}_b^\theta({\mathbb R}^d;{\mathbb R}^d)).
\]
Since the inverse operator $[\mathscr{I}_{-\gamma}]^{-1}=\mathscr{I}_{\gamma}$ maps ${\mathcal C}_b^{\theta}({\mathbb R}^d)$ isomorphically onto ${\mathcal C}_b^{\theta+\gamma}({\mathbb R}^d)$ whenever $\theta+\gamma\in (0,1)$, we obtain
\[
\nabla u_n\in \bigcap_{\tilde{\theta}\in (\alpha+\beta-1-\alpha/p,\alpha+\beta-1)}L^{\frac{p\alpha}{\alpha-p(\alpha+\beta-1-\tilde{\theta})}}([0,T];{\mathcal C}_b^{\tilde{\theta}}({\mathbb R}^d;{\mathbb R}^d)).
\]
Furthermore, for every $\tilde{\theta}\in (\alpha+\beta-1-\alpha/p,\alpha+\beta-1)$ there is a positive constant $C=C(d,\alpha,\beta,p,[b]_{p,\beta},\tilde{\theta})$ such that
\begin{equation}\label{2.45}
\|\nabla u_n\|_{p\alpha/[\alpha-p(\alpha+\beta-1-\tilde{\theta})],\tilde{\theta}} \leq
C(d,\alpha,\beta,p,\tilde{\theta},\|b\|_{p,\beta}) \|f\|_{p,\beta}.
\end{equation}

\eqref{2.45}, in conjunction with \eqref{2.39} and \eqref{2.35}, induces that
\begin{equation}\label{2.46}
\begin{split} u_n\in&  \bigcap_{0\leq \theta <\alpha+\beta-1-\alpha/p}L^\infty([0,T];{\mathcal C}_b^{1+\theta}({\mathbb R}^d))
\\ & \qquad\bigcap \bigcap_{\alpha+\beta-1-\alpha/p<\theta<\alpha+\beta-1}
L^{\frac{p\alpha}{\alpha-p(\alpha+\beta-1-\theta)}}([0,T];{\mathcal C}_b^{1+\theta}({\mathbb R}^d)),
\end{split}
\end{equation}
and, for every fixed $\theta_1\in (0,\alpha+\beta-1-\alpha/p)$ and $\theta_2\in (\alpha+\beta-1-\alpha/p,\alpha+\beta-1)$, there is a positive constant $C$ such that
\begin{equation}\label{2.47}
\|u_n\|_{\infty,1+\theta_1}+ \|\nabla u_n\|_{p\alpha/[\alpha-p(\alpha+\beta-1-\theta_2)],\theta_2}
\leq
 C\|f\|_{p,\beta}.
\end{equation}

Since $u_n$ satisfies \eqref{2.31},  relation \eqref{2.46} imlpies
\begin{equation}\label{2.48}
\partial_tu_n\in  \bigcap_{0\leq \theta <\alpha+\beta-1-\alpha/p}L^p([0,T];{\mathcal C}_b^\theta({\mathbb R}^d)) .
\end{equation}
Additionally, for every $\theta\in (0,\alpha+\beta-1-\alpha/p)$ there is a positive constant $C$,
depending on $d,T,\alpha$, $\beta,p,\theta,\lambda$ and $\|b\|_{p,\beta}$ such that
\begin{equation}\label{2.49}
\|\partial_tu_n\|_{p,\theta}\leq \|\Delta^{\frac{\alpha}{2}} u_n\|_{p,\theta}+\|\nabla u_n\|_{\infty,\theta}\|b_n\|_{p,\theta}+\lambda
\|u_n\|_{\infty,\theta}+ \|f\|_{p,\theta}\leq C\|f\|_{p,\beta}.
\end{equation}

Taking into account \eqref{2.47}, \eqref{2.49}, Morey's inequality (see \cite[Theorem 4, p.282]{Evans}) and Ascoli--Arzel\`{a}'s theorem, one can extract an (unlabelled) subsequence $u_n$  and a measurable function $u\in {\mathcal H}$ such that
\[
u_n(t,x)\rightarrow u(t,x), \ \ {\rm for every}  \ \ (t,x)\in [0,T]\times {\mathbb R}^d, \ \ {\rm as} \ \ n\rightarrow \infty.
\]
In addition, using \eqref{2.47}, \eqref{2.49} and Morey's inequality, one also has
\begin{equation*}
\begin{split}&\|\mathscr{I}_{-\gamma} u_n\|_{\mathcal{C}_b^{(1-\frac{1}{p})\wedge(\theta-\gamma)}([0,T]\times \mathbb{R}^d)}+\|\mathscr{I}_{-\gamma} u_n\|_{\infty,1}
\\ \leq&\;
C(d,T,\alpha,\beta,p,\lambda,\theta,\gamma,\|b\|_{p,\beta})\|f\|_{p,\beta}, \quad \forall \ 0<\gamma<\theta<\alpha+\beta-1-\frac{\alpha}{p}.
\end{split}
\end{equation*}
By interpolation inequalities for H\"{o}lder continuous functions and the Ascoli--Arzel\`{a} theorem,
it follows that
\[
\nabla u_n(t,x)\rightarrow \nabla u(t,x) \ \ {\rm (up \ to \ a \ unlabelled \ subsequence)},
\]
for every $(t,x)\in [0,T]\times {\mathbb R}^d$ as $n$ tends to infinity. Moreover, $\nabla u$ satisfies \eqref{2.20} in view of \eqref{2.34}. Estimates \eqref{2.39}, \eqref{2.47} and \eqref{2.49} also remain valid for $u$.

Finally, since each $u_n$ satisfies \eqref{2.31} and convergence \eqref{2.30} holds, the limiting function $u$ satisfies, for all $(t,x)\in [0,T]\times{\mathbb R}^d$, the following integral equation
\begin{equation}\label{2.50}
u(t,x)=\int_0^te^{-\lambda (t-s)}ds\int_{{\mathbb R}^d}K(t-s,x-y)[b(s,y)\cdot \nabla u(s,y)+f(s,y)]dy.
\end{equation}
Thus $u$ satisfies \eqref{2.1} (see \cite[Lemma 2.1]{TDW}).

\smallskip
$\bullet$ \textbf{(Uniqueness)} It is sufficient to show $u\equiv 0$ whenever $f$ vanishes since the equation is linear. Let $u$ be a strong solution of \eqref{2.1} with $f\equiv 0$. By \eqref{2.20} we get $\nabla u=0$, and then conclude $u\equiv 0$ by the integral representation \eqref{2.50}. \end{proof}

\begin{remark}\label{rem2.5}Let $q_1\in [1,\infty)$ and $\theta_1\in (0,1)$. If
\[
f\in L^{\infty}([0,T];\mathcal{C}_b(\mathbb{R}^d))\cap L^{q_1}([0,T];\mathcal{C}_b^{\theta_1}(\mathbb{R}^d)),
 \]
then
\[
f\in L^{q_2}([0,T];\mathcal{C}_b^{\theta_2}(\mathbb{R}^d)),
\]
for any $\theta_2\in  (0,\theta_1)$, where
$q_2=q_1\theta_1/\theta_2$. In this case, the following interpolation inequality holds
\begin{equation}\label{2.51}
[f]_{q_2,\theta_2} \leq 2^{\frac{\theta_1-\theta_2}{\theta_1}}\|f\|_{\infty,0}^{\frac{\theta_1-\theta_2}{\theta_1}}[f]_{q_1,\theta_1}^{\frac{\theta_2}{\theta_1}}.
\end{equation}
Let $u$ be the unique strong solution of the Cauchy problem \eqref{2.1} given in Theorem \ref{th2.4}. From \eqref{2.19}, we deduce that
\begin{equation}\label{2.52}
u\in \bigcap_{q\geq 1}
L^q([0,T];{\mathcal C}_b^{\alpha+\beta-\frac{\alpha}{p}}({\mathbb R}^d)), \quad  {\rm if} \ \ \alpha-1-\frac{\alpha}{p}\leq 0.
\end{equation}
Moreover, by combining \eqref{2.19} with \eqref{2.51}, for every $\theta\in (\alpha+\beta-1-\alpha/p,\alpha+\beta-1)$, we have
\begin{equation}\label{2.53}
[\nabla u]_{q,\alpha+\beta-1-\alpha/p}\leq 2^{\frac{\theta-\alpha-\beta+1+\alpha/p}{\theta}}\|\nabla u\|_{\infty,0}^{\frac{\theta-\alpha-\beta+1+\alpha/p}{\theta}}[\nabla u]_{p\alpha/[\alpha-p(\alpha+\beta-1-\theta)],\theta}^{\frac{\alpha+\beta-1-\alpha/p}{\theta}},
\end{equation}
where
\[
q=\frac{p\alpha\theta}{[\alpha-p(\alpha+\beta-1-\theta)](\alpha+\beta-1-\alpha/p)}>\frac{p\alpha}{
\alpha-p(\alpha+\beta-1-\theta)}.
\]
In view of \eqref{2.53}, along with \eqref{2.20}, \eqref{2.45} (valid also for $\nabla u$) and H\"{o}lder's inequality, it follows that for every $q>1$ there exist positive constants $\tilde{\varepsilon}=\tilde{\varepsilon}(p,\alpha,\beta,q)$ and $C=C(d,\alpha,\beta,p,q,\|b\|_{p,\beta})$ such that
\begin{equation}\label{2.54}
[\nabla u]_{q,\alpha+\beta-1-\alpha/p}\leq
C\lambda^{-\tilde{\varepsilon}}\|f\|_{p,\beta},
\end{equation}
for all large enough $\lambda$.
\end{remark}

\section{Proof of Theorem \ref{th1.1}}\label{sec3}\setcounter{equation}{0}
We carry this out into four parts: weak existence, uniqueness in law, pathwise uniqueness and Davie's type uniqueness. Note that if $b$ is bounded in the temporal variable, then it is $p$-integrable for every $p<\infty$. Throughout the following calculations, we always assume $p<\infty$.

\subsection{To prove Theorem \ref{th1.1} (i)}\label{sec3.1}\setcounter{equation}{0}
Let $\hat{\rho}$ be a nonnegative normalized mollifier in ${\mathbb R}^d$,
\[
0\leq \hat{\varrho} \in {\mathcal C}^\infty_0({\mathbb R}^d), \ \  {\rm supp}(\hat{\varrho})\subset B_1 \ \ {\rm and} \ \ \int_{{\mathbb R}^d}\hat{\varrho}(x)dx=1.
\]
For $n\in{\mathbb N}$, set $\hat{\varrho}_n(x)=n\hat{\varrho}(nx)$ and regularize $b$ in space by convolution with $\hat{\varrho}_n$,
\[
b^n(t,x)=(b(t,\cdot)\ast \hat{\varrho}_n)(x)=\int_{{\mathbb R}^d}b(t,x-y)\hat{\varrho}_n(y)dy.
\]
Then $b^n\in L^1([0,T];{\mathcal C}_b^1({\mathbb R}^d;{\mathbb R}^d))$ and
\begin{equation}\label{3.1}
\|b^n(t,\cdot)\|_0\leq \|b(t,\cdot)\|_0, \ a.e. \ t\in [0,T], \ \forall \ n\in {\mathbb N} \ \ \mbox{and} \ \
  \lim\limits_{n\rightarrow\infty}\|b^n-b\|_{1,0}=0.
\end{equation}

By the Cauchy--Lipschitz theorem, there is a unique
$\{\mathcal{F}_t\}_{t\in [0,T]}$-adapted, c\'{a}dl\`{a}g, $d$-dimensional process $\{X_t^n\}_{t\in [0,T]}$
defined on $(\Omega,\mathcal{F},\{\mathcal{F}_t\}_{t\in [0,T]},\mathbb{P})$ such that
\[
X_t^n=x+\int_0^tb^n(s,X^n_s)ds+L_t, \ \ {\mathbb P}-a.s..
\]
Define $\psi^n_t:=X^n_t-x-L_t$. Then it satisfies
\begin{equation}\label{3.2}
\psi^n_t=\int_0^tb^n(s,X_s^n)ds=\int_0^tb^n(s,x+\psi^n_s+L_s)ds.
\end{equation}

With the help of \eqref{3.1}, for every $0\leq t_1<t_2\leq T$,
\begin{equation}\label{3.3}
\sup_{n}\mathbb{E}\int_{t_1}^{t_2}|b^n(s,x+\psi^n_s+L_s)|ds\leq  \int_{t_1}^{t_2}\|b(s,\cdot)\|_0ds.
\end{equation}
Taking into account \eqref{3.2} and \eqref{3.3},  one concludes that
\begin{equation}\label{3.4}
\lim_{c\rightarrow \infty}\sup_n\sup_{0\leq t\leq T}\mathbb{P}\{|\psi_t^n|>c\}=0,
\end{equation}
and for every $\varsigma>0$,
\begin{equation}\label{3.5}
\lim_{\delta\downarrow 0}\sup_{n}
\sup_{0\leq t_1,t_2\leq T\atop{|t_1-t_2|\leq \delta}}\mathbb{P}\{|\psi_{t_1}^n-\psi_{t_2}^n|>\varsigma\}=0.
\end{equation}

From \eqref{3.4} and \eqref{3.5}, along with  Prohorov's theorem,
there exists a subsequence still denoted by itself such that
$\{(\psi_t^n,L_t)\}_{t\in [0,T]}$  converges weakly. Next, Skorohod's
representation theorem implies that there is a probability space
$(\tilde{\Omega},\tilde{\mathcal{F}},\{\tilde{\mathcal{F}}_t\}_{t\in [0,T]},\tilde{\mathbb{P}})$
and random processes $\{(\tilde{\psi}_t^n,\tilde{L}_t^n)\}_{t\in [0,T]}$, $\{(\tilde{\psi}_t,\tilde{L}_t)\}_{t\in [0,T]}$
on this probability space such that

\smallskip
$(1)$ the finite-dimensional distributions of $\{(\tilde{\psi}_t^n,\tilde{L}_t^n)\}_{t\in [0,T]}$
coincide with those of $\{(\psi_t^n,L_t)\}_{t\in [0,T]}$.

\smallskip
$(2)$ $(\tilde{\psi}_\cdot^n,\tilde{L}_\cdot^n)$  converges to $(\tilde{\psi}_\cdot,\tilde{L}_\cdot)$, \ $\tilde{\mathbb{P}}-a.s.$.

In particular, the process $\{\tilde{L}_t\}_{t\in [0,T]}$ remains a $d$-dimensional symmetric rotationally invariant $\alpha$-stable process and
\begin{equation}\label{3.6}
\tilde{\psi}_t^n=\int_0^tb^n(s,x+\tilde{\psi}_s^n+\tilde{L}_s^n)ds.
\end{equation}
Let $\tilde{\mathbb{E}}$ denote the expectation under the probability space $(\tilde{\Omega},\tilde{\mathcal{F}},\{\tilde{\mathcal{F}}_t\}_{t\in [0,T]},\tilde{\mathbb{P}})$. For $k\in{\mathbb N}$,
\begin{equation}\label{3.7}
\begin{split}&\tilde{\mathbb{E}}\int_0^T|b^n(s,x+\tilde{\psi}_s^n+\tilde{L}_s^n)-
b(s,x+\tilde{\psi}_s+\tilde{L}_s)|ds\\ \leq&\; \tilde{\mathbb{E}}\int_0^T|b^n(s,x+\tilde{\psi}_s^n+\tilde{L}_s^n)-
b^k(s,x+\tilde{\psi}_s^n+\tilde{L}_s^n)|ds
\\&+\tilde{\mathbb{E}}\int_0^T|b^k(s,x+\tilde{\psi}_s^n+\tilde{L}_s^n)-
b^k(s,x+\tilde{\psi}_s+\tilde{L}_s)|ds
\\&+\tilde{\mathbb{E}}\int_0^T|b^k(s,x+\tilde{\psi}_s+\tilde{L}_s)-
b(s,x+\tilde{\psi}_s+\tilde{L}_s)|ds
\\ \leq&\; C\Big[ \|b^n-b^k\|_{1,0}+
\|b^k-b\|_{1,0}\Big]+
\tilde{\mathbb{E}}\int_0^T|b^k(s,x+\tilde{\psi}_s^n+\tilde{L}_s^n)-
b^k(s,x+\tilde{\psi}_s+\tilde{L}_s)|ds.
\end{split}
\end{equation}

First letting $n\rightarrow \infty$, and then $k\rightarrow \infty$, relation \eqref{3.1} together with \eqref{3.7} yields
\[
\lim_{n\rightarrow \infty}\int_0^tb^n(s,x+\tilde{\psi}_s^n+\tilde{L}_s^n)ds=\int_0^tb(s,x+\tilde{\psi}_s+\tilde{L}_s)ds.
\]
Therefore,
\[
\tilde{\psi}_t=\int_0^tb(s,x+\tilde{\psi}_s+\tilde{L}_s)ds.
\]
Defining $\tilde{X}_t:=\psi_t+x+\tilde{L}_t$, the process $\{\tilde{X}_t\}_{t\in [0,T]}$ is
$\{\tilde{\mathcal{F}}_t\}_{t\in [0,T]}$-adapted, c\'{a}dl\`{a}g and satisfies \eqref{1.14}. Thus there is a weak solution to \eqref{1.1}.

\subsection{To prove Theorem \ref{th1.1} (ii)}\label{sec3.2}
Let $f\in {\mathcal C}_0^\infty([0,T]\times {\mathbb R}^d)$ and consider the following backward Cauchy problem
\begin{equation}\label{3.8}
\left\{\begin{array}{ll}
\partial_{t}u(t,x)+\Delta^{\frac{\alpha}{2}}u(t,x)+b(t,x)\cdot \nabla u(t,x)=f(t,x), \ (t,x)\in [0,T)\times {\mathbb R}^d, \\ [0.1cm]
u(T,x)=0, \  x\in {\mathbb R}^d.
\end{array}\right.
\end{equation}
By Theorem \ref{th2.4}, there is a unique strong solution $u$ to \eqref{3.8}.  Moreover, if $\alpha-1-\alpha/p\leq 0$, then $u\in {\mathcal H}$; otherwise $u$ belongs to the space given in \eqref{2.21}.

Let $\tilde{\varrho}$ be a nonnegative normalized mollifier in ${\mathbb R}^{d+1}$ with
\[
{\rm supp}\tilde{\varrho}\subset \{(t,x) \ | -1\leq t\leq 0, \, |x|\leq 1\}.
\]
For $k\in {\mathbb N}$, set $\tilde{\varrho}_k(t,x)=k^{d+1}
\tilde{\varrho}(kt,kx)$ and define
\[
\tilde{b}_k(t,x)=b\ast \tilde{\varrho}_k(t,x), \ \tilde{f}_k(t,x)=f\ast \tilde{\varrho}_k(t,x) \  {\rm and} \ \tilde{u}_k(t,x)=u\ast \tilde{\varrho}_k(t,x).
\]
Then,
\begin{equation}\label{3.9}
\left\{
  \begin{array}{ll}
 \lim\limits_{k\rightarrow\infty}\tilde{u}_k(t,x)=u(t,x),  \ \ \forall \ (t,x)\in [0,T]\times{\mathbb R}^d,\\ [0.2cm]
   \lim\limits_{k\rightarrow\infty}\nabla \tilde{u}_k(t,x)=\nabla u(t,x),  \ \ \forall \ (t,x)\in [0,T]\times{\mathbb R}^d, \\ [0.2cm]
 \lim\limits_{k\rightarrow\infty}\tilde{b}_k(t,x)=b(t,x), \ \ a.e. \ t\in [0,T], \ \forall \ x\in {\mathbb R}^d, \\ [0.2cm]
 \lim\limits_{k\rightarrow\infty}\tilde{f}_k(t,x)=f(t,x), \ \ a.e. \ t\in [0,T], \ \forall \ x\in {\mathbb R}^d,
  \end{array}
\right.
\end{equation}
and
\begin{equation}\label{3.10}
\partial_t\tilde{u}_k(t,x)+\Delta^{\frac{\alpha}{2}}\tilde{u}_k(t,x)+(b\cdot \nabla u)_k(t,x)=-\tilde{f}_k(t,x), \ (t,x)\in [0,T-1/k)\times {\mathbb R}^d.
\end{equation}

Let $X_t$ be a weak solution of \eqref{1.1}. Applying It\^{o}'s formula  (see \cite[Theorem 4.4.10]{App}) to $\tilde{u}_k(t,X_t)$, and then letting $k$ tend to infinity, leads to the limiting equation
\begin{equation}\label{3.11}
\begin{split}&u(t,X_t)-u(0,x)=-\int_0^tf(s,X_s)ds \\ &\qquad + \int_0^t\int_{{\mathbb R}^d\setminus \{0\}}[u(s,X_{s-}+z)-u(s,X_{s-})] \tilde{N}(ds,dz),  \ \forall \ t\in [0,T],
\end{split}
\end{equation}
where \eqref{3.9} and \eqref{3.10} have been employed in the passage to the limit.

Taking $t=T$ in \eqref{3.11} and then expectations yields
\begin{equation}\label{3.12}
u(0,x)={\mathbb E}\int_0^Tf(s,X_s)ds.
\end{equation}
Since the left-hand side of \eqref{3.12} is independent of $X_t$,  uniqueness in law follows.

\subsection{To prove Theorem \ref{th1.1} (iii)}\label{sec3.3}
 Let us consider the following vector-valued backward Cauchy problem for large $\lambda>0$
\begin{equation}\label{3.13}
\left\{\begin{array}{ll}
\partial_{t}U(t,x)+\Delta^{\frac{\alpha}{2}}U(t,x)+b(t,x)\cdot \nabla U(t,x)-\lambda U(t,x)\\ \qquad\qquad\qquad =-b(t,x), \ (t,x)\in [0,T)\times {\mathbb R}^d, \\
U(T,x)=0, \  x\in {\mathbb R}^d.
\end{array}\right.
\end{equation}
In view of Theorem \ref{th2.4}, there exists a unique strong solution $U$ to \eqref{3.13}.  Moreover, if $\alpha-1-\alpha/p\leq 0$, then $U\in {\mathcal H}^d$; otherwise
\begin{equation}\label{3.14}
\left\{
\begin{array}{ll}
\!\!U\in \Big[L^p([0,T];{\mathcal C}_b^{\alpha+\beta}({\mathbb R}^d))\bigcap L^\infty([0,T];{\mathcal C}_b^{\alpha+\beta-\frac{\alpha}{p}}({\mathbb R}^d))\Big]^d, \  {\rm if} \ \alpha+\beta<2 \ {\rm or} \ \alpha+\beta-\frac{\alpha}{p}>2, \\ \!\! U\in \Big[\bigcap\limits_{0<\theta <1} L^p([0,T];{\mathcal C}_b^{2-\theta}({\mathbb R}^d))\bigcap L^\infty([0,T];{\mathcal C}_b^{\alpha+\beta-\frac{\alpha}{p}}({\mathbb R}^d))\Big]^d, \ \  {\rm if}  \ \alpha+\beta=2, \\ \!\! U\in \Big[ L^p([0,T];{\mathcal C}_b^{\alpha+\beta}({\mathbb R}^d)) \bigcap \bigcap\limits_{0<\theta <1} L^\infty([0,T];{\mathcal C}_b^{2-\theta}({\mathbb R}^d))\Big]^d, \ \ {\rm if}  \ \alpha+\beta-\frac{\alpha}{p}=2.
 \end{array}
\right.
\end{equation}
Furthermore, in both cases there exist positive constants $\varepsilon=\varepsilon(p,\alpha,\beta)$ and $C=C(d,\alpha,\beta,p,\|b\|_{p,\beta})$ such that
\begin{equation}\label{3.15}
\sup_{(t,x)\in [0,T]\times {\mathbb R}^d}|\nabla U(t,x)|\leq C\lambda^{-\varepsilon}\leq \frac{1}{2},
\end{equation}
where the last inequality holds for sufficiently large $\lambda$.

Let $\tilde{\varrho}_k$ and $\tilde{b}_k$ be defined in Sec. \ref{sec3.2}, and set $\tilde{U}_k(t,x)=U\ast \tilde{\varrho}_k(t,x)$. Applying It\^{o}'s formula to $\tilde{U}_k(t,X_t)$, and then letting $k$ tends to infinity, yields an analogue of \eqref{3.11}
\begin{equation*}
\begin{split}&U(t,X_t)-U(0,x)
 \\   =&\;\lambda\int_0^tU(s,X_s)ds -\int_0^tb(s,X_s)ds + \int_0^t\int_{{\mathbb R}^d\setminus \{0\}}[U(s,X_{s-}+z)-U(s,X_{s-})] \tilde{N}(ds,dz)  \\ =&\;x+L_t-X_t +\lambda\int_0^tU(s,X_s)ds + \int_0^t\int_{{\mathbb R}^d\setminus \{0\}}[U(s,X_{s-}+z)-U(s,X_{s-})] \tilde{N}(ds,dz),
\end{split}
\end{equation*}
where the second identity uses the fact that $X$ satisfies \eqref{1.1}. Thus,
\begin{equation}\label{3.16}
\begin{split}
X_t=&-U(t,X_t)+U(0,x)+x+L_t+\lambda\int_0^tU(s,X_s)ds
 \\ & + \int_0^t\int_{{\mathbb R}^d\setminus \{0\}}[U(s,X_{s-}+z)-U(s,X_{s-})] \tilde{N}(ds,dz).
\end{split}
\end{equation}

Let $X_t(x)$ and $Y_t(y)$ be two weak solutions of \eqref{1.1} starting from $x$ and $y$, respectively, defined on the same probability space. From \eqref{3.15} and \eqref{3.16}, it follows that
\begin{equation}\label{3.17}
\begin{split}&|X_t(x)-Y_t(y)|
\\ \leq&\; \|\nabla U\|_{\infty,0}\bigg[|X_t-Y_t|+|x-y|+\lambda\int_0^t|X_s-Y_s|ds\bigg] +|x-y|  \\ & + \bigg|\int_0^t\int_{{\mathbb R}^d\setminus \{0\}}[U(s,X_{s-}+z)-U(s,X_{s-}) -U(s,Y_{s-}+z)+U(s,Y_{s-})] \tilde{N}(ds,dz)\bigg| \\ \leq&\;  \frac{1}{2}|X_t-Y_t|+\frac{3}{2}|x-y|+\frac{\lambda}{2}\int_0^t|X_s-Y_s|ds  + \bigg|\int_0^t\int_{{\mathbb R}^d\setminus \{0\} }[U(s,X_{s-}+z)  \\ & -U(s,X_{s-})-U(s,Y_{s-}+z) +U(s,Y_{s-})] \tilde{N}(ds,dz)\bigg|.
\end{split}
\end{equation}

By Kunita's first inequality (see \cite[Theorem 4.4.23]{App}), inequality \eqref{3.17} implies
\begin{equation}\label{3.18}
\begin{split}&{\mathbb E}\sup_{0\leq r\leq t}|X_r(x)-Y_r(y)|^2 \\ \leq&\; C|x-y|^2+C\int_0^t{\mathbb E}\sup_{0\leq r\leq s}|X_r-Y_r|^2ds \\ & + C{\mathbb E}\int_0^tds\int_{{\mathbb R}^d\setminus \{0\}}|U(s,X_{s-}+z)-U(s,X_{s-})-U(s,Y_{s-}+z)+U(s,Y_{s-})|^2\nu(dz)
 \\ =&\; C|x-y|^2+C\int_0^t{\mathbb E}\sup_{0\leq r\leq s}|X_r-Y_r|^2ds \\ & + C{\mathbb E}\int_0^tds\int_{{\mathbb R}^d\setminus \{0\}}|U(s,X_s+z)-U(s,X_s)-U(s,Y_s+z)+U(s,Y_s)|^2\nu(dz)
 \\ \leq&\; C|x-y|^2+C\int_0^t{\mathbb E}\sup_{0\leq r\leq s}|X_r-Y_r|^2ds  \\ &+ C{\mathbb E}\int_0^tds\int_{0<|z|< 1}|U(s,X_s+z)-U(s,X_s)-U(s,Y_s+z)+U(s,Y_s)|^2\nu(dz),
\end{split}
\end{equation}
where in the fifth line we have used the fact $\{X_t\}_{t\in [0,T]}$ and $\{Y_t\}_{t\in [0,T]}$ are L\'{e}vy processes and every L\'{e}vy process has at most countably discontinuities (see \cite[Theorems 2.1.8 and 2.9.2]{App}).

\smallskip
\textbf{Case 1:} $\alpha-1-\alpha/p\leq 0$.  Notice that
\begin{equation}\label{3.19}
 U \in
\bigg[\bigcap_{\alpha+\beta-1-\alpha/p<\theta<\alpha+\beta-1}
 L^{\frac{p\alpha}{\alpha-p(\alpha+\beta-1-\theta)}}([0,T];{\mathcal C}_b^{1+\theta}({\mathbb R}^d))\bigg]^d.
\end{equation}
Hence,
\begin{equation}\label{3.20}
\begin{split}&|U(s,X_s+z)-U(s,X_s)-U(s,Y_s+z)+U(s,Y_s)|  \\ \leq&\;\sum_{i,j=1}^d\sup_{0\leq \tau \leq 1}|\partial_{x_i} U_j(s,\tau X_s+(1-\tau)Y_s+z)-\partial_{x_i} U_j(s,\tau X_s+(1-\tau) Y_s)||X_s-Y_s| \\ \leq&\;
d^2[\nabla U(s,\cdot)]_{\alpha+\beta-1-\epsilon\alpha/p}|z|^{\alpha+\beta-1-\frac{\epsilon\alpha}{p}}
|X_s-Y_s|,
\end{split}
\end{equation}
where $0<\epsilon<[p(\alpha/2+\beta-1)/\alpha]\wedge 1$.

In view of \eqref{3.18} and \eqref{3.20}, we arrive at
\begin{equation}\label{3.21}
\begin{split}&{\mathbb E}\sup_{0\leq r\leq t}|X_r(x)-Y_r(y)|^2-C|x-y|^2-C\int_0^t{\mathbb E}\sup_{0\leq r\leq s}|X_r-Y_r|^2ds  \\ \leq&\; C \int_0^t[\nabla U(s,\cdot)]_{\alpha+\beta-1-\epsilon\alpha/p}^2{\mathbb E}\sup_{0\leq r\leq s}|X_r-Y_r|^2ds
\int_{0<|z|<1}|z|^{\alpha+2\beta-2-\frac{2\epsilon\alpha}{p}-d}dz
  \\ \leq&\; C \int_0^t[\nabla U(s,\cdot)]_{\alpha+\beta-1-\epsilon\alpha/p}^2{\mathbb E}\sup_{0\leq r\leq s}|X_r-Y_r|^2ds,
\end{split}
\end{equation}
where in the last inequality we used $1-\alpha/2<\beta$ and $0<\epsilon<p(\alpha/2+\beta-1)/\alpha$.

Moreover,
\[
\frac{p\alpha}{\alpha-p(\alpha+\beta-1-\theta)}\Big|_{\theta=\alpha+\beta-1-\frac{\epsilon\alpha}{p}}=\frac{p}{1-\epsilon},
\]
Hence, from \eqref{3.19},
\begin{equation}\label{3.22}
[\nabla U(s,\cdot)]_{\alpha+\beta-1-\epsilon\alpha/p} \in
L^{\frac{p}{1-\epsilon}}([0,T]).
\end{equation}
If $p\geq 2$, then
\begin{equation}\label{3.23}
[\nabla U(s,\cdot)]_{\alpha+\beta-1-\epsilon\alpha/p}^2 \in  L^1([0,T]).
\end{equation}
Otherwise, choosing $\epsilon=1-p/2<[p(\alpha/2+\beta-1)/\alpha]\wedge 1$ in \eqref{3.22}, also leads to \eqref{3.23}.

Combining \eqref{3.21}, \eqref{3.23} and Gr\"{o}nwall's inequality gives
\begin{equation}\label{3.24}
{\mathbb E}\sup_{0\leq r\leq T}|X_r(x)-Y_r(y)|^2\leq C|x-y|^2,
\end{equation}
which establishes pathwise uniqueness.

\smallskip
\textbf{Case 2:} $\alpha-1-\alpha/p>0$. By \eqref{3.14}, an analogue of \eqref{3.20} holds
\begin{equation*}
\begin{split}&|U(s,X_s+z)-U(s,X_s)-U(s,Y_s+z)+U(s,Y_s)|  \\  \leq&\; \left\{
\begin{array}{ll}
d^2[\nabla U(s,\cdot)]_{\alpha+\beta-1}|z|^{\alpha+\beta-1}
|X_s-Y_s|, & {\rm if} \ \alpha+\beta< 2, \\ [0.2cm] d^2[\nabla U(s,\cdot)]_{\alpha+\beta-1-\epsilon}|z|^{1-\epsilon}
|X_s-Y_s|, \ & {\rm if}  \ \alpha+\beta=2,
\\  [0.2cm] d^2\|\nabla^2 U(s,\cdot)\|_0|z||X_s-Y_s|,\ & {\rm if}  \ \alpha+\beta>2,
 \end{array}
\right.
\end{split}
\end{equation*}
where $\epsilon$ can take any values in $(0,1)$. Consequently, an analogue of \eqref{3.21} is obtained
\begin{equation}\label{3.25}
\begin{split}
&{\mathbb E}\sup_{0\leq r\leq t}|X_r(x)-Y_r(y)|^2-C|x-y|^2-C\int_0^t{\mathbb E}\sup_{0\leq r\leq s}|X_r-Y_r|^2ds
 \\ \leq&\;
C\left\{\begin{array}{ll}
 \int_0^t[\nabla U(s,\cdot)]_{\alpha+\beta-1}^2{\mathbb E}\sup\limits_{0\leq r\leq s}|X_r-Y_r|^2ds
\int_0^1\tau^{\alpha+2\beta-3}d\tau, & {\rm if} \ \alpha+\beta< 2, \\ [0.2cm] \int_0^t[\nabla U(s,\cdot)]_{\alpha+\beta-1-\epsilon}^2{\mathbb E}\sup\limits_{0\leq r\leq s}|X_r-Y_r|^2ds
\int_0^1\tau^{\beta-2\epsilon-1}d\tau, \ & {\rm if}  \ \alpha+\beta=2,
\\ [0.2cm] \int_0^t\|\nabla^2 U(s,\cdot)\|_0^2{\mathbb E}\sup\limits_{0\leq r\leq s}|X_r-Y_r|^2ds
\int_0^1\tau^{1-\alpha}d\tau,\ & {\rm if}  \ \alpha+\beta>2,
 \end{array}
\right.
 \\  \leq&\;
C\left\{\begin{array}{ll}
 \int_0^t[\nabla U(s,\cdot)]_{\alpha+\beta-1}^2{\mathbb E}\sup\limits_{0\leq r\leq s}|X_r-Y_r|^2ds, & {\rm if} \ \alpha+\beta< 2, \\ [0.2cm] \int_0^t[\nabla U(s,\cdot)]_{\alpha+\beta-1-\epsilon}^2{\mathbb E}\sup\limits_{0\leq r\leq s}|X_r-Y_r|^2ds, \ & {\rm if}  \ \alpha+\beta=2,
 \\ [0.2cm] \int_0^t\|\nabla^2 U(s,\cdot)\|_0^2{\mathbb E}\sup\limits_{0\leq r\leq s}|X_r-Y_r|^2ds,\ & {\rm if}  \ \alpha+\beta>2,
 \end{array}
\right.
\end{split}
\end{equation}
where in the fifth line of \eqref{3.25} we used $\beta>1-\alpha/2$ and in the sixth line we have chosen $\epsilon<\beta/2$.

Since $\alpha-1-\alpha/p>0$, we have $\alpha>1$ and $p>\alpha/(\alpha-1)>2$. Therefore, the functions
\[
[\nabla U(s,\cdot)]_{\alpha+\beta-1}^2, \ \ [\nabla U(s,\cdot)]_{\alpha+\beta-1-\epsilon}^2 \ \ {\rm and} \ \ \|\nabla U(s,\cdot)\|_0^2
\]
are Lebesgue integrable on $[0,T]$. With the aid of Gr\"{o}nwall's inequality, \eqref{3.24} remains valid that proves pathwise uniqueness.

\subsection{To prove Theorem \ref{th1.1} (iv)}\label{sec3.4}
Now, we verify the Davie type uniqueness for \eqref{1.1}. For every $0\leq \tau\leq t\leq T$, consider the following SDE
\begin{equation}\label{3.26}
dX_{\tau,t}(x)=b(t,X_{\tau,t}(x))dt+dL_t, \quad 0\leq \tau \leq t\leq T, \quad X_{\tau,\tau}(x)=x\in {\mathbb R}^d.
\end{equation}
There exists a unique strong solution to \eqref{3.26}. Let $X_{\tau,t}(x)$ and $Y_{\tau,t}(y)$ be two solutions of \eqref{3.26} starting from $x$ and $y$, respectively. From \eqref{3.16},
\begin{equation}\label{3.27}
\begin{split}X_{\tau,t}(x)-Y_{\tau,t}(y) =&-U(t,X_{\tau,t})+U(t,Y_{\tau,t})+U(\tau,x)-U(\tau,y)+x-y \\&+\lambda\int_\tau^t
[U(s,X_{\tau,s})-U(s,Y_{\tau,s})]ds
 + \int_\tau^t\int_{{\mathbb R}^d\setminus \{0\}}[U(s,X_{\tau,s-}+z) \\ & -U(s,X_{\tau,s-})-U(s,Y_{\tau,s-}+z)+U(s,Y_{\tau,s-})]  \tilde{N}(ds,dz).
\end{split}
\end{equation}

Fix $\hat{q}\geq 2$ (to be specified later). By \eqref{3.27} and Kunita's inequality (see \cite[Theorem 4.4.23]{App}), we obtain
\begin{equation*}\label{3.28}
\begin{split}&{\mathbb E}\sup_{\tau\leq r\leq t}|X_{\tau,r}(x)-Y_{\tau,r}(y)|^{\hat{q}} \\ \leq&\; C|x-y|^{\hat{q}}+C\int_\tau^t{\mathbb E}\sup_{\tau\leq r\leq s}|X_{\tau,r}-Y_{\tau,r}|^{\hat{q}}ds
+C{\mathbb E}\bigg(\int_\tau^tds\int_{|z|\geq1}|X_{\tau,s}-Y_{\tau,s}|^2\nu(dz)\bigg)^{\frac{\hat{q}}{2}}  \\ & +C{\mathbb E}\int_\tau^tds\int_{|z|\geq 1}|X_{\tau,s}-Y_{\tau,s}|^{\hat{q}}\nu(dz) +  C{\mathbb E}\bigg(\int_\tau^tds\int_{0<|z|<1}|U(s,X_{\tau,s}+z) \\ &-U(s,X_{\tau,s})-U(s,Y_{\tau,s}+z)+U(s,Y_{\tau,s})|^2
\nu(dz)\bigg)^{\frac{\hat{q}}{2}}  \\ & +C{\mathbb E}\int_\tau^tds\int_{0<|z|<1}|U(s,X_{\tau,s}+z)-U(s,X_{\tau,s})-U(s,Y_{\tau,s}+z)+
U(s,Y_{\tau,s})|^{\hat{q}}\nu(dz).
\end{split}
\end{equation*}
If one uses H\"{o}lder's inequality, then
\begin{equation}\label{3.28}
\begin{split} &{\mathbb E}\sup_{\tau\leq r\leq t}|X_{\tau,r}(x)-Y_{\tau,r}(y)|^{\hat{q}} \\
\leq&\;
C|x-y|^{\hat{q}}+C\int_\tau^t{\mathbb E}\sup_{\tau\leq r\leq s}|X_{\tau,r}-Y_{\tau,r}|^{\hat{q}}ds +  C{\mathbb E}\bigg(\int_\tau^tds\int_{0<|z|<1}|U(s,X_{\tau,s}+z) \\ &-U(s,X_{\tau,s})-U(s,Y_{\tau,s}+z)+
U(s,Y_{\tau,s})|^2\nu(dz)\bigg)^{\frac{\hat{q}}{2}}  \\ & + C{\mathbb E}\int_\tau^tds\int_{0<|z|< 1}|U(s,X_{\tau,s}+z)-U(s,X_{\tau,s})-U(s,Y_{\tau,s}+z)+U(s,Y_{\tau,s}) |^{\hat{q}}\nu(dz)
 \\ =:&\; C|x-y|^{\hat{q}}+C\int_\tau^t{\mathbb E}\sup_{\tau\leq r\leq s}|X_{\tau,r}-Y_{\tau,r}|^{\hat{q}}ds + J_1(\tau,t)+J_2(\tau,t),
\end{split}
\end{equation}
where the constant $C$ depends only on $d,\alpha,\beta,p,\hat{q}$ and $\|b\|_{p,\beta}$.

\medskip
\textbf{Case 1:} $\alpha-1-\alpha/p\leq 0$.  Let $0<\epsilon<[p(\alpha/2+\beta-1)/\alpha]\wedge 1$. From \eqref{3.19} and \eqref{3.20}, it follows that
\begin{equation}\label{3.29}
\begin{split}J_1(\tau,t)\leq&\; C(d,\alpha,\beta,\epsilon,p,\hat{q},\|b\|_{p,\beta}) {\mathbb E}\bigg( \int_\tau^t[\nabla U(s,\cdot)]_{\alpha+\beta-1-\epsilon\alpha/p}^2
|X_{\tau,s}-Y_{\tau,s}|^2ds \\
&\times \int_{0<|z|<1}|z|^{\alpha+2\beta-2-\frac{2\epsilon\alpha}{p}-d}dz\bigg)^{\frac{\hat{q}}{2}}
 \\
\leq&\; C(d,\alpha,\beta,\epsilon,p,\hat{q},\|b\|_{p,\beta})\bigg(\int_0^T[\nabla U(s,\cdot)]_{\alpha+\beta-1-\epsilon\alpha/p}^2ds\bigg)^{\frac{\hat{q}}{2}} \\
&\times {\mathbb E}\sup_{\tau\leq r\leq t}|X_{\tau,r}(x)-Y_{\tau,r}(y)|^{\hat{q}}.
\end{split}
\end{equation}
Moreover, by taking $\theta=\alpha+\beta-1-\epsilon\alpha/(2p)$ in \eqref{3.19}, it turns out that
\begin{equation}\label{3.30}
[\nabla U(s,\cdot)]_{\alpha+\beta-1-\epsilon\alpha/(2p)} \in
L^{\frac{2p}{2-\epsilon}}([0,T]).
\end{equation}

If $p\geq 2$, then by \eqref{3.23} and \eqref{3.30}, both $[\nabla U(s,\cdot)]_{\alpha+\beta-1-\epsilon\alpha/p}$ and $[\nabla U(s,\cdot)]_{\alpha+\beta-1-\epsilon\alpha/(2p)}$ belong to $L^2([0,T])$. Furthermore, in view of H\"{o}lder's inequality, \eqref{2.51}, \eqref{3.15} and \eqref{3.30}, we infer that
\begin{equation}\label{3.31}
\begin{split}[\nabla U]_{2,\alpha+\beta-1-\epsilon\alpha/p} \leq&\;C(T)[\nabla U]_{(2\alpha+2\beta-2-\epsilon\alpha/p)/(\alpha+\beta-1-\epsilon\alpha/p),\alpha+\beta-1-\epsilon\alpha/p}
 \\[6pt] \leq&\;
C(T)[\nabla U]_{2,\alpha+\beta-1-\epsilon\alpha/(2p)}^{\frac{\alpha+\beta-1-\epsilon\alpha/p}{\alpha+\beta-1-\epsilon\alpha/(2p)}} \|\nabla U\|_{\infty,0}^{\frac{\epsilon\alpha/(2p)}{\alpha+\beta-1-\epsilon\alpha/(2p)}}
 \\[6pt]  \leq&\;C(T)[\nabla U]_{2p/(2-\epsilon),\alpha+\beta-1-\epsilon\alpha/(2p)}^{\frac{\alpha+\beta-1-\epsilon\alpha/p}{\alpha+\beta-1-\epsilon\alpha/(2p)}} \|\nabla U\|_{\infty,0}^{\frac{\epsilon\alpha/(2p)}{\alpha+\beta-1-\epsilon\alpha/(2p)}}
 \\[6pt]
\leq&\; C(d,T,\alpha,\beta,\epsilon,p,\|b\|_{p,\beta}) \lambda^{\frac{-\epsilon\alpha \varepsilon/(2p)}{\alpha+\beta-1-\epsilon\alpha/(2p)}}.
\end{split}
\end{equation}
This, in conjunction with \eqref{3.29}, indicates
\begin{equation}\label{3.32}
J_1(\tau,t)\leq C(d,T,\alpha,\beta,\epsilon,p,\hat{q},\|b\|_{p,\beta})
\lambda^{\frac{-\epsilon\alpha \varepsilon\hat{q} /(2p)}{\alpha+\beta-1-\epsilon\alpha/(2p)}} {\mathbb E}\sup_{\tau\leq r\leq t}|X_{\tau,r}(x)-Y_{\tau,r}(y)|^{\hat{q}}.
\end{equation}
Finally, for fixed $\hat{q}$ and $\epsilon$, choosing $\lambda$ large enough ensures that
\begin{equation}\label{3.33}
J_1(\tau,t)\leq \frac{1}{2} {\mathbb E}\sup_{\tau\leq r\leq t}|X_{\tau,r}(x)-Y_{\tau,r}(y)|^{\hat{q}}.
\end{equation}

If $1<p<2$, choose
\[
\epsilon=1-\frac{p}{2}-\hat{\theta}, \ \
0<\hat{\theta}<\Big[\frac{p(\alpha/2+\beta-1)}{\alpha}-1+\frac{p}{2}\Big]\wedge\Big(1-\frac{p}{2}\Big),
\]
which guarantees that
\[
0<\epsilon<\Big[\frac{p(\alpha/2+\beta-1)}{\alpha}\Big]\wedge 1 \quad {\rm for}
\ \ p>\frac{\alpha}{\alpha+\beta-1}.
\]
From \eqref{3.22}, we derive
\begin{equation}\label{3.34}
[\nabla U(s,\cdot)]_{\alpha+\beta-1-\epsilon\alpha/p}=[\nabla U(s,\cdot)]_{3\alpha/2+\beta-1-\alpha/p+\alpha\hat{\theta}/p}\in
L^{\frac{2p}{p+2\hat{\theta}}}([0,T]).
\end{equation}
Since $U$ is the unique strong solution of the Cauchy problem \eqref{3.13}, with the help of Remark \ref{rem2.2} (see \eqref{2.52} and \eqref{2.54}),
\begin{equation}\label{3.35}
U\in \bigg[\bigcap_{q\geq 1}
L^{2q}([0,T];{\mathcal C}_b^{\alpha+\beta-\frac{\alpha}{p}}({\mathbb R}^d))\bigg]^d,
\end{equation}
and there are positive constants $\tilde{\varepsilon}=\tilde{\varepsilon}(p,\alpha,\beta,q)$ and $C=C(d,\alpha,\beta,p,q,\|b\|_{p,\beta})$ such that for every large enough $\lambda$
\begin{equation}\label{3.36}
[\nabla U]_{2q,\alpha+\beta-1-\alpha/p}\leq
C\lambda^{-\tilde{\varepsilon}}.
\end{equation}

Fix $q>1$ and set
\[
\eta=\frac{p(q-1)}{q(p+2\hat{\theta})-p}\in(0,1).
\]
Then $\eta(1/2+\hat{\theta}/p)<1/p$ for $\hat{\theta}<1-p/2$.
By the interpolation inequality,
\[
[\nabla U]_{2,\alpha+\beta-1-\alpha/p+\eta\alpha(1/2+\hat{\theta}/p)}\leq [\nabla U]_{2p/(p+2\hat{\theta}),\alpha+\beta-1-\alpha/p+\alpha(1/2+\hat{\theta}/p)}^\eta [\nabla U]_{2q,\alpha+\beta-1-\alpha/p}^{1-\eta},
\]
together with \eqref{3.34}--\eqref{3.36}, it follows that
\begin{equation}\label{3.37}
[\nabla U(s,\cdot)]_{\alpha+\beta-1-\alpha/p+\eta\alpha(1/2+\hat{\theta}/p)} \in  L^2([0,T])
\end{equation}
and moreover,
\begin{equation}\label{3.38}
[\nabla U]_{2,\alpha+\beta-1-\alpha/p+\eta\alpha(1/2+\hat{\theta}/p)}
\leq C(d,\alpha,\beta,\hat{\theta},p,q,\|b\|_{p,\beta}) \lambda^{-\tilde{\varepsilon}(1-\eta)}.
\end{equation}
Now take $\epsilon=1-p\eta/2-\eta\hat{\theta}$, so that
\[
\alpha+\beta-1-\frac{\epsilon\alpha}{p}=\alpha+\beta-1-\frac{\alpha}{p}+\eta\alpha\Big(\frac{1}{2}+\frac{\hat{\theta}}{p}\Big).
\]
Since $\eta \rightarrow p/(p+2\hat{\theta})$ as $q\rightarrow \infty$, and given that $p>\alpha/(\alpha+\beta-1)$, one can choose $q$ big enough to ensure
 $\epsilon<[p(\alpha/2+\beta-1)/\alpha]\wedge 1$. Therefore, we get \eqref{3.33} from \eqref{3.29} and \eqref{3.38} for large enough $\lambda$.

On the other hand, for all $p\in (1,\infty)$,  $J_2$ can be estimated as
\begin{equation}\label{3.39}
\begin{split}&J_2(\tau,t) \\ \leq&\; C {\mathbb E} \int_\tau^t[\nabla U(s,\cdot)]_{\alpha+\beta-1-\alpha/p}^{\hat{q}}
|X_{\tau,s}-Y_{\tau,s}|^{\hat{q}}ds\int_{0<|z|<1}|z|^{(\alpha+\beta-1)\hat{q}-\frac{\alpha\hat{q}}{p}-d-\alpha}dz 
\end{split}
\end{equation}
Taking $\hat{q}>\alpha/(\alpha+\beta-1-\alpha/p)$, we get $(\alpha+\beta-1-\alpha/p)\hat{q}-1-\alpha>-1$. Whence,
\begin{equation}\label{3.40}
J_2(\tau,t)\leq C {\mathbb E} \int_\tau^t[\nabla U(s,\cdot)]_{\alpha+\beta-1-\alpha/p}^{\hat{q}}
|X_{\tau,s}-Y_{\tau,s}|^{\hat{q}}ds.
\end{equation}

In conjunction with \eqref{3.28}, \eqref{3.33} and \eqref{3.40}, for every fixed $\hat{q}>\alpha/(\alpha+\beta-1-\alpha/p)$, and for $\lambda$ big enough, one has
\begin{equation*}
\begin{split}&{\mathbb E}\sup_{\tau\leq r\leq t}|X_{\tau,r}(x)-Y_{\tau,r}(y)|^{\hat{q}}
 \\ \leq&\; C|x-y|^{\hat{q}}+C  \int_\tau^t\Big[1+[\nabla U(s,\cdot)]_{\alpha+\beta-1-\alpha/p}^{\hat{q}}\Big]{\mathbb E}\sup_{\tau\leq r\leq s} |X_{\tau,r}-Y_{\tau,r}|^{\hat{q}}ds.
\end{split}
\end{equation*}
Consequently, making use of \eqref{3.35} together with Gr\"{o}nwall's inequality and H\"{o}lder's inequality, it follows that
\begin{equation}\label{3.41}
\sup_{0\leq \tau\leq T}{\mathbb E}\sup_{\tau\leq r\leq t}|X_{\tau,r}(x)-Y_{\tau,r}(y)|^{\hat{q}} \leq C|x-y|^{\hat{q}},
\end{equation}
for every $\hat{q}\geq 1$.

\smallskip
\textbf{Case 2:} $\alpha-1-\alpha/p>0$. The proof procedures can be carried out by four cases:
$\alpha+\beta<2$, $\alpha+\beta=2$, $\alpha+\beta-\alpha/p=2$ and $\alpha+\beta-\alpha/p>2$. Since the reasoning is analogous, we present the details only for the first subcase.

From \eqref{3.14} and the interpolation inequality
\[
[f]_{p/(1-\epsilon/2),\alpha+\beta-1-\epsilon \alpha/2p} \leq  [f]_{p,\alpha+\beta-1}^{\frac{2-\epsilon}{2}}[f]_{\infty,\alpha+\beta-1-\alpha/p}^{\frac{\epsilon}{2}}, \ \ \forall \ \epsilon\in (0,1)
\]
we have
\begin{equation}\label{3.42}
\nabla U \in
\bigg[\bigcap_{0<\epsilon<1}
L^{\tfrac{2p}{2-\epsilon}}\big([0,T];\mathcal{C}_b^{\alpha+\beta-1-\tfrac{\epsilon\alpha}{2p}}(\mathbb{R}^d)
\big)\bigg]^{d\times d}.
\end{equation}
Since $\alpha>1$ and $p>\alpha/(\alpha-1)>2$, the combination of \eqref{3.31} and \eqref{3.42} yields \eqref{3.32}, which in turn implies \eqref{3.33} for big enough $\lambda$.

Moreover, as $U\in [L^\infty([0,T];{\mathcal C}_b^{\alpha+\beta-\alpha/p}({\mathbb R}^d))]^d$, one first derives an analogue of \eqref{3.39} for large $\hat{q}$, and subsequently obtains the analogue of \eqref{3.40}
\begin{equation}\label{3.43}
J_2(\tau,t)\leq C [\nabla U]_{\infty,\alpha+\beta-1-\alpha/p}^{\hat{q}}{\mathbb E} \int_\tau^t
|X_{\tau,s}-Y_{\tau,s}|^{\hat{q}}ds.
\end{equation}
Taking into account \eqref{3.28}, \eqref{3.33}, \eqref{3.43},  and applying Gr\"{o}nwall's inequality together with H\"{o}lder's inequality, one recovers estimate \eqref{3.41}.

In addition, for any $\tilde{\alpha}\in (0,\alpha)$,
\begin{equation}\label{3.44}
\int_{|z|\geq1}|z|^{\tilde{\alpha}}\nu(dz)=c(d,\alpha)\int_{|z|\geq1}|z|^{\tilde{\alpha}-d-\alpha}dz<\infty.
\end{equation}
Therefore, by \eqref{3.41}, \eqref{3.44} and \cite[Theorem 1.1]{Priola18}, Davie's type uniqueness for \eqref{1.1} is established.

\begin{remark}\label{rem3.1} From \eqref{3.29} and \eqref{3.38}, relation \eqref{3.33} follows directly provided that
\[
\frac{p(\alpha/2+\beta-1)}{\alpha}\geq1,  \ \ \forall \ \epsilon\in (0,1).
\]
However, if
\[
\frac{p(\alpha/2+\beta-1)}{\alpha}<1 \ \ {\rm and} \ \ \epsilon\in \Big(\frac{p(\alpha/2+\beta-1)}{\alpha},1\Big),
\]
then
\[
\int_{0<|z|<1}|z|^{\alpha+2\beta-2-\frac{2\epsilon\alpha}{p}-d}dz= \int_0^1r^{\alpha+2\beta-3-\frac{2\epsilon\alpha}{p}}dr=\infty.
\]
Thus, the right hand side of \eqref{3.29} diverges. Even though \eqref{3.37} and \eqref{3.38} remain valid, we must further verify that $\epsilon<p(\alpha/2+\beta-1)/\alpha$.
\end{remark}

\section{Proof of Theorem \ref{th1.3}}\label{sec4}\setcounter{equation}{0}
Our approach to constructing two distinct solutions is inspired by Tanaka, Tsuchiya and Watanabe \cite{TTW}, who established nonuniqueness for a time-independent drift under the conditions $d=1$ and $\alpha+\beta<1$. The construction here is actually one-dimensional and extends trivially to higher dimensions by taking $b=(b_1,b_2,\ldots,b_d)$ with $b_i\equiv 0$ for $i\geq 2$. Therefore, without loss of generality, we restrict to the case $d=1$.

For $\alpha\in (0,2)$, $\beta\in (0\vee (1-\alpha),1)$ and $1\leq p<\alpha/(\alpha+\beta-1)$, let $p<\hat{p}<\alpha/(\alpha+\beta-1)$ and define $b(t,x)=t^{-\frac{1}{\hat{p}}}\hat{h}(x)$ where
\begin{equation}\label{4.1}
\hat{h}(x)=\left\{\begin{array}{ll}
{\rm sign}(x)|x|^\beta, & {\rm if} \ |x|\leq\theta_0, \\ [0.1cm] h(x), \ & {\rm if}  \ |x|>\theta_0,
 \end{array}
\right.
\end{equation}
with $\theta_0>0$, $h\in \mathcal{C}^1_b(\mathbb{R})$ satisfying $h(\pm\theta_0)=\pm\theta_0^\beta$ and $h^\prime\geq 0$. Then $b\in L^p([0,T];{\mathcal C}_b^{\beta}({\mathbb R}))$ and $b$ is increasing in the spatial variable. For such a drift, we construct two distinct solutions to the following SDE
\begin{equation}\label{4.2}
X_t=\int_0^tb(s,X_s)ds+L_t, \quad t\in [0,T],
\end{equation}
i.e. \eqref{1.1} with zero initial data.

Choose a decreasing sequence $\{h_n\}_{n\geq 1}$ of uniformly bounded, Lipschitz continuous functions that decreases to $\hat{h}$ as $n\rightarrow \infty$. By the Cauchy--Lipschitz theorem, there exists a unique strong solution $X_t^n$ to the following SDE
\[
X_t^n=\int_0^ts^{-\frac{1}{\hat{p}}}h_n(X_s^n)ds+L_t, \quad t\in [0,T].
\]
For $n\geq 2$, the difference of $X_t^n$ and $X_t^{n-1}$ satisfies
\[
\frac{d}{dt}[X_t^n-X_t^{n-1}]=t^{-\frac{1}{\hat{p}}}[h_n(X_t^n)-h_{n-1}(X_t^{n-1})], \quad t\in [0,T].
\]
Thus $X_\cdot^n-X_\cdot^{n-1}\in W^{1,1}([0,T])$, ${\mathbb P}$-a.s..

Consider the function $f(x)=x_{+}=(x+|x|)/2$, which is Lipschitz continuous. With the help of the chain rule (see \cite[Theorem 2.2]{Bou01}),
\begin{equation}\label{4.3}
\begin{split}\frac{d}{dt}[X_t^n-X_t^{n-1}]_+=&t^{-\frac{1}{\hat{p}}}\mbox{sign}([X_t^n-X_t^{n-1}]_+)[h_n(X_t^n)-h_{n-1}(X_t^{n-1})] \\ \leq&\; t^{-\frac{1}{\hat{p}}}\mbox{sign}([X_t^n-X_t^{n-1}]_+)[h_{n-1}(X_t^n)-h_{n-1}(X_t^{n-1})] \\ \leq&\; t^{-\frac{1}{\hat{p}}}C_{n-1}[X_t^n-X_t^{n-1}]_+ \quad t\in [0,T],
\end{split}
\end{equation}
where the first inequality follows from $h_n\leq h_{n-1}$, while the second from the Lipschitz continuity of $h_{n-1}$ with Lipschitz constant $C_{n-1}$. Applying  the Gr\"{o}nwall inequality to \eqref{4.3} yields $[X_t^n-X_t^{n-1}]_+=0$, which suggests  $X_t^n\leq X_t^{n-1}$ for all $t\in [0,T]$, ${\mathbb P}$-a.s.. Hence, the sequence $\{X^n_t(\omega)\}_{n\geq 1}$ is decreasing, and by the monotone convergence theorem, converges to a limit as $n\rightarrow \infty$. Since each $X^n_t$ is ${\mathcal F}_t$-measurable, the limit, denoted by $X_{\max,t}$, is also ${\mathcal F}_t$-measurable and satisfies \eqref{4.2}.

Let $X=\{X_t\}_{t\in [0,T]}$ be a solution of \eqref{4.2}. Substituting $X$ and $b$ for $X^n$ and $b_n$ in \eqref{4.3}, respectively, leads to
\[
X_t\leq X^{n-1}_t, \quad \forall \ n\geq 2, \ t\in [0,T], \ {\mathbb P}-a.s..
\]
Therefore, $X_{\max}$ is the biggest solution among those of \eqref{4.2}.

We now show that this solution is independent of the choice of the decreasing sequence $\{h_n\}_{n\geq 1}$. Let $\{\tilde{h}_n\}_{n\geq 1}$ be another decreasing sequence of uniformly bounded Lipschitz continuous functions converging to $\hat{h}$ as $n\rightarrow \infty$, and let $\tilde{X}_t^n$ be the unique solution of the following SDE
\begin{equation}\label{4.4}
\tilde{X}_t^n=\int_0^ts^{-\frac{1}{\hat{p}}}\tilde{h}_n(\tilde{X}_s^n)ds+L_t, \quad t\in [0,T].
\end{equation}
Then $\tilde{X}^n_t(\omega)$ decreases to a limit as $n\rightarrow \infty$; denote this limit by $\tilde{X}_t(\omega)$.  The process $\tilde{X}$ is also a solution of \eqref{4.2}. Thus $\tilde{X}_t\leq X_{\max,t}$. By a similar argument to \eqref{4.3}, one has $\tilde{X}_t\geq X_{\max,t}$. Therefore, $\tilde{X}_t=X_{\max,t}$ and  we call $X_{\max,t}$ the maximum solution of \eqref{4.2}.

Analogously, let $\{\hat{h}_n\}_{n\geq 1}$ be an increasing sequence of uniformly bounded Lipschitz functions converging to $\hat{h}$ as $n\rightarrow \infty$. The corresponding limit $X_{\min}$ is the smallest solution among those of \eqref{4.2} and remains independent of the choice of the increasing sequence.  We call this solution as the minimum solution of \eqref{4.2}.

Let us estimate the maximum and minimum solutions of \eqref{4.2}. By Khinchin's law of the iterated logarithm for $\alpha$-stable processes,
\[
\limsup_{t\rightarrow 0}\frac{|L_t|}{t^{\frac{1}{\alpha}}(\log(\log(\frac{1}{t})))}=C(\alpha), \quad {\mathbb P}-a.s..
\]
which implies
\begin{equation}\label{4.5}
\lim_{t\rightarrow 0}[|L_t|t^{-\frac{1}{\alpha}+\epsilon}] =0, \quad {\mathbb P}-a.s., \quad \forall \ \epsilon>0.
\end{equation}
Let $C_0\in (0,1)$ be a fixed and small enough real number, and let $\delta=(1-1/\hat{p})/(1-\beta)<1/\alpha$. Define the event $\Omega_0$ by
\[
\Omega_0=\{\omega; \ \exists \ T_0(\omega)<T \ {\rm such \ that} \ |L_t(\omega)|\leq C_0t^{\delta}, \ \ \forall \ t\in [0,T_0(\omega)]\}.
\]
By \eqref{4.5},  $\mathbb{P}(\Omega_0)=1$.

Choose two positive constants $C_1$ and $C_2$ such that $C_1^\beta/\delta-C_0>C_1$ and $C_2^\delta C_1\leq \theta_0$. For $\omega\in \Omega_0$, define $Y_t$ by
\begin{equation}\label{4.6}
Y_t(\omega)=\left\{\begin{array}{ll}
C_1t^{\delta}, & {\rm if} \ t\in [0,T_0(\omega)\wedge C_2], \\ [0.1cm] C_1(T_0(\omega)\wedge C_2)^{\delta}, & {\rm if}  \ t\in [T_0(\omega)\wedge C_2,T],
 \end{array}
\right.
\end{equation}
and set
$$
\int_0^tb(s,Y_s(\omega))ds+L_t(\omega)=:Z_t(\omega).
$$
For $\omega\in \Omega_0$ and $t\in [0,T_0(\omega)\wedge C_2]$, straightforward calculation gives
\begin{equation}\label{4.7}
\begin{split}Z_t(\omega)=&C_1^\beta\int_0^ts^{-\frac{1}{\hat{p}}}s^{\beta\delta}ds+L_t(\omega) \\ \geq&
\frac{C_1^\beta t^{\beta \delta+1-1/\hat{p}}}{\beta \delta+1-1/\hat{p}}-C_0t^\delta  =(C_1^\beta/\delta-C_0)t^\delta
\geq C_1t^\delta=Y_t(\omega).
\end{split}
\end{equation}

Let $X_{\max}$ be the maximum solution of \eqref{4.2}. For $\omega\in \Omega_0$ and $t\in [0,T_0(\omega)\wedge C_2]$, we have
\begin{equation}\label{4.8}
X_{\max,t}(\omega)=\lim_{n\rightarrow \infty}X_t^n(\omega)=
\lim_{n\rightarrow \infty}\int_0^ts^{-\frac{1}{\hat{p}}}h_n(X_s^n(\omega))ds,
\end{equation}
where $\{h_n\}_{n\geq 1}$ is a decreasing sequence of uniformly bounded, Lipschitz continuous functions such that $h_n$ decreases to $\hat{h}$ as $n\rightarrow \infty$, and  $h_n^\prime\geq 0$ for every fixed $n$.

For every fixed $n$ and $\omega\in \Omega_0$, we get an analogue of \eqref{4.3} that
\begin{equation*}
\begin{split}
\frac{d}{dt}[Z_t(\omega)-X_t^n(\omega)]_+=&t^{-\frac{1}{\hat{p}}}\mbox{sign}([Z_t(\omega)-X_t^n(\omega)]_+)
[\hat{h}(Y_t(\omega))-h_n(X_t^n(\omega))]  \\ \leq&\; t^{-\frac{1}{\hat{p}}}\mbox{sign}([Z_t(\omega)-X_t^n(\omega)]_+)[h_n(Z_t(\omega))-h_n(X_t^n(\omega))] \\ \leq&\; t^{-\frac{1}{\hat{p}}}C_n[Z_t(\omega)-X_t^n(\omega)]_+ \quad t\in [0,T_0(\omega)\wedge C_2].
\end{split}
\end{equation*}
This, together with Gr\"{o}nwall's inequality, \eqref{4.7} and \eqref{4.8}, leads to
\[
Y_t(\omega)\leq Z_t(\omega)\leq X_{\max,t}(\omega), \quad \forall \ t\in [0,T_0(\omega)\wedge C_2] \ \ {\rm and} \ \  \omega\in \Omega_0.
\]

Similarly, define $\tilde{Y}_t$ by
\begin{equation}\label{4.9}
\tilde{Y}_t(\omega)=\left\{\begin{array}{ll}
-C_1t^{\delta}, & {\rm if} \ t\in [0,T_0(\omega)\wedge C_2], \\ [0.1cm] -C_1(T_0(\omega)\wedge C_2)^{\delta}, & {\rm if}  \ t\in [T_0(\omega)\wedge C_2,T],
 \end{array}
\right.
\end{equation}
and one verifies that
\begin{equation*}
\begin{split}\int_0^tb(s,\tilde{Y}_s(\omega))ds+L_t(\omega)=&-C_1^\beta\int_0^ts^{-\frac{1}{\hat{p}}}s^{\beta\delta}ds
+L_t(\omega) \\ \leq&\;
-\frac{C_1^\beta t^{\beta \delta+1-1/\hat{p}} }{\beta \delta+1-1/\hat{p}}+C_0t^\delta\leq -C_1t^\delta=\tilde{Y}_t(\omega),
\end{split}
\end{equation*}
for all $t\in [0,T_0(\omega)\wedge C_2]$. Thus the minimum solution $X_{\min}$ satisfies
\[
X_{\min,t}(\omega)\leq \tilde{Y}_t(\omega), \quad t\in [0,T_0(\omega)\wedge C_2] \ {\rm and} \ \omega\in \Omega_0.
\]

By \eqref{4.6} and \eqref{4.9}, it follows that
\[
X_{\max}\neq X_{\min}.
\]
Hence the maximum and minimum solutions of \eqref{4.2} are distinct.

\section{Conclusions}
In this work we studied time-inhomogeneous stochastic differential equations driven by
symmetric rotationally invariant $\alpha$-stable processes with $\alpha \in (0,2)$ and with
drift coefficients belonging to subcritical Lebesgue--H\"{o}lder spaces
$L^p([0,T];{\mathcal C}_b^{\beta}({\mathbb R}^d;{\mathbb R}^d))$. Our analysis combined
probabilistic and analytic techniques, including Prohorov's theorem, Skorohod's
representation, regularity estimates for fractional Kolmogorov equations, and
It\^{o}--Tanaka's trick.

We established the weak well-posedness of such SDEs under the condition
$\beta \in (0,1)$, $\alpha+\beta>1$, and $p>\alpha/(\alpha+\beta-1)$. Furthermore,
we proved pathwise uniqueness and Davie's type uniqueness when
$\beta>1-\alpha/2$, thereby extending previous results to the fully subcritical regime.
These findings provide a unified framework for the regularisation-by-noise phenomenon
for SDEs driven by $\alpha$-stable processes, covering both Brownian motion ($\alpha=2$)
and jump-driven cases $\alpha \in (0,2)$.

To complement our positive results, we constructed a counterexample showing that
pathwise uniqueness may fail for drifts in the supercritical regime. This illustrates
the sharpness of our conditions and highlights the intrinsic limitations of
regularisation effects by stable processes in the presence of highly irregular drifts.


\begin{thebibliography}{99}\addtolength{\itemsep}{-1.0ex}
\bibitem{App}  Applebaum, D.: L\'{e}vy Processes and Stochastic Calculus. \emph{Cambridge University Press}, Cambridge, 2009.

\bibitem{ABM} Athreya, S., Butkovsky, O. and Mytnik, L.: Strong existence and uniqueness for stable stochastic differential equations with distributional drift. \emph{Ann. Probab.} \textbf{48}(1), (2020), 178--210.


\bibitem{BFGM}  Beck, A.,  Flandoli, F., Gubinelli, M. and Maurelli, M.: Stochastic ODEs and stochastic linear PDEs with critical drift: regularity, duality and uniqueness.  \emph{Electron. J. Probab.}
    \textbf{24}(136), (2019), 1--72.

\bibitem{BP} Bogachev, V.I. and  Pilipenko, A.Y.: Strong solutions to stochastic equations with L\'{e}vy noise and a discontinuous drift coefficient. \emph{Dokl. Math.} \textbf{94}(1), (2016), 438--440.

\bibitem{Bou01} Bouchut, F.: Renormalized solutions to the Vlasov equation with coefficients of bounded variation. \emph{Arch. Rational. Mech. Anal.} \textbf{157}(1), (2001), 75--90.

\bibitem{BG} Butkovsky, O. and Gallay, S.:  Weak existence for SDEs with singular drifts and fractional
Brownian or L\'{e}vy noise beyond the subcritical regime, arXiv:2311.12013.

\bibitem{CKS} Chen, Z.Q., Kim, P. and Song, R.: Dirichlet heat kernel estimates for fractional Laplacian with gradient perturbation. \emph{Ann. Probab.} \textbf{40}(6), (2012), 2483--2538.

\bibitem{CZ} Chen, Z.Q. and Zhang, X.: Heat kernels and analyticity of non-symmetric jump
diffusion semigroups. \emph{Probab. Theory Relat. Fields} \textbf{165}(1-2), (2016), 267--312.

\bibitem{CSZ} Chen, Z.Q., Song, R. and Zhang, X.: Stochastic flows for L\'{e}vy processes with H\"{o}lder drifts. \emph{Rev. Mat. Iberoam.} \textbf{34}(4), (2018), 1755--1788.

\bibitem{CZZ} Chen, Z.Q., Zhang, X. and Zhao, G.: Supercritical SDEs driven by multiplicative stable-like L\'{e}vy processes. \emph{Trans. Amer. Math. Soc.} \textbf{374}(11), (2021), 7621--7655.

\bibitem{Dav07} Davie, A.M.: Uniqueness of solutions of stochastic differential  equations. \emph{Int. Math. Res. Not.} \textbf{24}, (2007), Art. ID rnm124, 26.

\bibitem{DM}  De Raynal, P.E. and Menozzi, S.: On multidimensional stable-driven stochastic differential equations with Besov drift. \emph{Electron. J. Probab.} \textbf{27}(52), (2022), 1--52.

\bibitem{DMP1} De Raynal, P.E., Menozzi, S. and Priola, E.: Weak well-posedness of multidimensional stable driven SDEs in the critical case. \emph{Stoch. Dyn.} \textbf{20}(6), (2020), 2040004.

\bibitem{DMP2} De Raynal, P.E., Menozzi, S. and  Priola, E.:  Schauder estimates for drifted fractional operators in the supercritical case. \emph{J. Funct. Anal.} \textbf{278}(8), (2020), 108425.

\bibitem{Evans} Evans, L.: Partial Differential Equations. \emph{Providence}, RI: Amer. Math. Soc., 2016.



\bibitem{FIZ} Fang, S., Imkeller, P. and Zhang, T.: Global flows for stochastic differential equations without global Lipschitz conditions. \emph{Ann. Probab.} \textbf{35}(1), (2007), 180--205.

\bibitem{FT} Fang, S. and Tian, R.: Some remarks on It\^{o} stochastic processes. \emph{Discrete Contin. Dyn. Syst. Ser. S}, \textbf{18}(10), (2025), 2702--2716.

\bibitem{FZ} Fang, S. and Zhang, T.: A study of a class of stochastic differential
equations with non-Lipschitzian coefficients. \emph{Probab. Theory Relat. Fields} \textbf{132}(3), (2005), 356--390.

\bibitem{FGP1} Flandoli, F., Gubinelli, M. and Priola, E.:  Well-posedness of the transport equation by stochastic perturbation. \emph{Invent. Math.} \textbf{180}(1), (2010), 1--53.

\bibitem{FGP2} Flandoli, F., Gubinelli, M. and Priola, E.: Flow of diffeomorphisms for SDEs with unbounded H\"{o}lder continuous drift. \emph{Bull. Sci. Math.} \textbf{134}(4), (2010), 405--422.

\bibitem{Fournier} Fournier, N.: On pathwise uniqueness for stochastic differential equations driven by stable L\'{e}vy processes. \emph{Ann. Inst. Henri Poincar\'{e} Probab. Stat.} \textbf{49}(1), (2013), 138--159.

\bibitem{GG} Galeati, L. and Gerencs\'{e}r,  M.:  Solution theory of fractional SDEs in complete subcritical regimes. \emph{Forum Math. Sigma} \textbf{13}(12), (2025), 1--66.


\bibitem{HP} Haadem, S. and Proske, F.: On the construction and Malliavin differentiability of solutions of L\'{e}vy noise driven SDE's with singular coefficients. \emph{J. Funct. Anal.} \textbf{266}(8), (2014), 5321--5359.

\bibitem{HW} Hao, Z. and Wu, M.: SDE driven by cylindrical $\alpha$-stable process with distributional drift and Euler's approximation, arXiv:2305.18139.

\bibitem{KS} Kim, P. and Song, R.: Stable process with singular drift. \emph{Stochastic Process. Appl.} \textbf{124}(7), (2014), 2479--2516.

\bibitem{KM1} Kinzebulatov, D. and Madou, K.R.: Stochastic equations with time-dependent singular drift. \emph{J. Differ. Equ.} \textbf{337}, (2022), 255--293.

\bibitem{KM2} Kinzebulatov, D. and Madou, K.R.: Strong solutions of SDEs with singular (form-bounded) drift via R\"{o}ckner--Zhao approach. \emph{Stoch. Dyn.} \textbf{24}(7), (2024), 2450050.

\bibitem{KSS} Kinzebulatov, D., Semenov, Y.A. and Song, R.: Stochastic transport equation with singular drift. \emph{Ann. Inst. Henri Poincar\'{e} Probab. Stat.} \textbf{60}(1), (2024), 731--752.

\bibitem{KP} Kremp, H. and Perkowski, N.: Multidimensional SDE with distributional drift and L\'{e}vy noise. \emph{Bernoulli} \textbf{28}(3), (2022), 1757--1783.

\bibitem{Krylov02} Krylov, N.V.: The Calder\'{o}n--Zygmund theorem and parabolic equations in
$L_p(\mathbb{R};\mathcal{C}^{2+\alpha})$.  \emph{Ann. Scuola Norm. Sci.} \textbf{1}(5), (2002), 799--820.


\bibitem{Kry21-1} Krylov, N.V.: On stochastic equations with drift in $L_d$. \emph{Ann. Probab.} \textbf{49}(5), (2021), 2371--2398.

\bibitem{Kry21-2} Krylov, N.V.: On strong solutions of Ito's equations with $a\in W^1_d$ and $b\in L_d$, \emph{Ann. Probab.} \textbf{49}(6), (2021), 3142--3167.

\bibitem{Kry21-3} Krylov, N.V.: Elliptic equations with VMO $a,b\in L_d$ and $c\in L_{d/2}$. \emph{Trans. Amer. Math. Soc.} \textbf{374}(4), (2021), 2805--2822.

\bibitem{Kry23-1} Krylov, N.V.:  On strong solutions of It\^{o}'s equations with $D\sigma$ and $b$ in Morrey classes containing $L_d$. \emph{Ann. Probab.} \textbf{51}(5), (2023), 1729--1751.

\bibitem{Kry23-2} Krylov, N.V.: Elliptic equations in Sobolev spaces with Morrey drift and the zeroth-order coefficients. \emph{Trans. Amer. Math. Soc.} \textbf{376}(10), (2023), 7329--7351.

\bibitem{Kry25} Krylov, N.V.: On weak and strong solutions of time inhomogeneous It\^{o}'s equations with VMO diffusion and Morrey drift. \emph{Stochastic Process. Appl.} \textbf{179}, (2025), 104505.

\bibitem{KR} Krylov, N.V. and R\"{o}ckner, M.:  Strong solutions to stochastic equations with singular time dependent drift. \emph{Probab. Theory Relat. Fields} \textbf{131}(2), (2005), 154--196.

\bibitem{Kulik} Kulik, A.: On weak uniqueness and distributional properties of a solution to an SDE with $\alpha$-stable noise. \emph{Stochastic Process. Appl.} \textbf{129}(2), (2019), 473--506.

\bibitem{Kurenok} Kurenok, V.P.: Stochastic equations with time-dependent drift
driven by L\'{e}vy processes. \emph{J. Theor. Probab.} \textbf{20}(4), (2007), 859--869.

\bibitem{LZ} Ling, C. and Zhao. G.: Nonlocal elliptic equation in H\"{o}lder space and the martingale problem.  \emph{J. Differ. Equ.} \textbf{314}, (2022), 653--699.

\bibitem{MNP} Mohammed, S.E., Nilssen, T.K. and Proske, F.N.: Sobolev differentiable stochastic flows for SDEs with singular coefficients: applications to the transport equation. \emph{Ann. Probab.} \textbf{43}(3), (2015), 1535--1576

\bibitem{MW} Mytnik, L. and Weinberger, J.:  Strong existence and uniqueness for singular SDEs driven by stable processes, arXiv:2404.13729.

\bibitem{Nam} Nam, K.:  Stochastic differential equations with critical drifts. \emph{Stochastic Process. Appl.} \textbf{130}(9), (2020), 5366--5393.

\bibitem{Priola12} Priola, E.: Pathwise uniqueness for singular SDEs driven by stable processes. \emph{Osaka J. Math.} \textbf{49}(2), (2012), 421--447.

\bibitem{Priola15} Priola, E.: Stochastic flow for SDEs with jumps and irregular drift term. \emph{Banach Center Publ.} \textbf{105}, (2015), 193--210.

\bibitem{Priola18} Priola, E.: Davie's type uniqueness for a class of SDEs with jumps. \emph{Ann. Inst. Henri Poincar\'{e} Probab. Stat.} \textbf{54}(2), (2018), 694--725.

\bibitem{RZ} R\"{o}ckner, M. and Zhao, G.: SDEs with critical time dependent drifts: strong solutions. \emph{Probab. Theory Relat. Fields} \textbf{192}, (2025), 1071--1111.


\bibitem{SX} Song, R. and Xie, L.: Weak and strong well-posedness of critical and
supercritical SDEs with singular coefficients. \emph{J. Differ. Equ.} \textbf{362}, (2023), 266--313.

\bibitem{Stein} Stein, E.: Singular Integrals and Differentiaility Properties of Functions. \emph{Princeton University Press}, Princeton, 1970.

\bibitem{TTW} Tanaka, H., Tsuchiya,  M. and Watanabe, S.: Perturbation of drift-type for L\'{e}vy processes. \emph{J. Math. Kyoto Univ.} \textbf{14}(1), (1974), 73--92.

\bibitem{TDW} Tian, R., Ding, L. and Wei, J.:  Strong solutions of stochastic differential equations with square integrable drift. \emph{Bull. Sci. Math.} \textbf{174}, (2022), 103085.

\bibitem{TW} Tian, R. and Wei, J.: Fractional Fokker--Planck--Kolmogorov equations with H\"{o}lder continuous drift. \emph{Fract. Calc. Appl. Anal.} \textbf{27}(5), (2024), 2456--2481.

\bibitem{TWD} Tian, R., Wei, J. and Duan, J.:  Stochastic differential equations with H\"{o}lder--Dini drift and driven by $\alpha$-stable processes. \emph{Stoch. Dyn.} \textbf{24}(5), (2024), 2450037.

\bibitem{Ver} Veretennikov, A.J.: On the strong solutions of stochastic differential equations. \emph{Theory Probab. Appl.} \textbf{24}(2), (1979), 354--366.

\bibitem{WDGL}  Wei, J.,  Duan, J., Gao, H. and Lv, G.:  Stochastic regularization for transport equations. \emph{Stoch. Partial Differ. Equ., Anal. Comput.} \textbf{9}(1), (2021), 105--141.

\bibitem{WHY1} Wei, J., Hu, J. and Yuan, C.: Stochastic differential equations with low regularity growing drifts and applications. \emph{SIAM J. Math. Anal.} \textbf{57}(5), (2025), 4867--4907.


\bibitem{WLW} Wei, J., Lv, G. and  Wang, W.: Stochastic transport equation with bounded and Dini continuous drift. \emph{J. Differ. Equ.} \textbf{323}(5), (2022), 359--403.


\bibitem{WLWu} Wei, J., Lv, G. and Wu, J.L.:  Stochastic differential equations with critically irregular drift coefficients. \emph{J. Differ. Equ.} \textbf{371}, (2023), 1--30.


\bibitem{XX}  Xia, P., Xie, L.,  Zhang, X. and Zhao, G.: $L^q(L^p)$-theory of stochastic differential equations. \emph{Stochastic Process. Appl.} \textbf{130}(8), (2020), 5188--5211.

\bibitem{XZZ} Xiong, J., Zheng, J. and Zhou, X.:  Unique strong solutions of L\'{e}vy processes driven stochastic differential equations with discontinuous coefficients. \emph{Stochastics} \textbf{91}(4), (2019), 592--612.

\bibitem{ZY} Zhang, S.Q. and  Yuan, C.: A Zvonkin's transformation for stochastic differential equations with singular drift and applications. \emph{J. Differ. Equ.} \textbf{297}, (2021), 277--319.

\bibitem{Zha05} Zhang, X.:  Strong solutions of SDEs with singular drift and Sobolev diffusion coefficients. \emph{Stochastic Process. Appl.} \textbf{115}(11), (2005), 1805--1818.

\bibitem{Zha11} Zhang, X.: Stochastic homeomorphism flows of SDEs with singular drifts
and Sobolev diffusion coefficients. \emph{Electron. J. Probab.} \textbf{16}(13), (2011), 1096--1116.

\bibitem{Zhang13} Zhang, X.: Stochastic differential equations with Sobolev drifts and driven
by $\alpha$-stable processes. \emph{Ann. Inst. Henri Poincar\'{e} Probab. Stat.} \textbf{49}(4), (2013), 1057--1079.

\bibitem{Zvo} Zvonkin, A.K.: A transformation of the phase space of a diffusion process
that removes the drift. \emph{Mat. Sb.} \textbf{93}(135), (1974), 491--525.
\end{thebibliography}
\end{document}